\documentclass[11pt]{amsart}
\usepackage{amsmath,amssymb,latexsym,soul,cite,mathrsfs}

\usepackage{color,enumitem,graphicx}
\usepackage[colorlinks=true,urlcolor=blue,
citecolor=red,linkcolor=blue,linktocpage,pdfpagelabels,
bookmarksnumbered,bookmarksopen]{hyperref}
\usepackage[english]{babel}
  \usepackage{graphicx}
  \usepackage[dvipsnames]{xcolor}
  \usepackage{trig}
  \usepackage{color}
  \usepackage{graphics}
  \definecolor{Mygrey}{gray}{0.75}
  \definecolor{MyFcja}{gray}{0.3}

\usepackage[left=2.9cm,right=2.9cm,top=2.8cm,bottom=2.8cm]{geometry}
\usepackage[hyperpageref]{backref}

\usepackage[colorinlistoftodos]{todonotes}
\makeatletter
\providecommand\@dotsep{5}
\def\listtodoname{List of Todos}
\def\listoftodos{\@starttoc{tdo}\listtodoname}
\makeatother

\numberwithin{equation}{section}

\def\R{\mathbb{R}}
\def\N{\mathbb{N}}
\newcommand{\hs}{\hspace{0.1cm}}
\newcommand{\disp}{\displaystyle}

\newtheorem{theorem}{Theorem}[section]
\newtheorem{proposition}[theorem]{Proposition}
\newtheorem{lemma}[theorem]{Lemma}
\newtheorem{corollary}[theorem]{Corollary}

\newtheorem{claim}[theorem]{Claim}
\newtheorem{remark}[theorem]{Remark}

\title[The method of the energy function and applications]
{The method of the energy function and applications}

\author[C. O. Alves]{Claudianor O. Alves}
\author[T. L. C. Coelho]{Tiago L. C. Coelho}
\author[J. R. Santos Jr.]{Jo\~ao R. Santos J\'unior$^\ast$}

\address[C. O. Alves]
{\newline\indent Unidade Acad\^{e}mica de Matem\'atica - UAMAT
\newline\indent
Universidade Federal de Campina Grande
\newline\indent
CEP: 58109-970, Campina Grande - Pb, Brazil.}
\email{\href{mailto: coalves@mat.ufcg.edu.br}{coalves@mat.ufcg.edu.br}
}

\address[T. L. C. Coelho]{\newline\indent Programa de Doutorado em Matem\'atica
\newline\indent
Instituto de Ci\^{e}ncias Exatas e Naturais
\newline\indent
Universidade Federal do Par\'a
\newline\indent
Avenida Augusto corr\^{e}a 01, 66075-110, Bel\'em, PA, Brazil}
\email{\href{mailto: tiagolcoelhoc@gmail.com}{tiagolcoelhoc@gmail.com}}

\address[J. R. Santos Jr.]{\newline\indent Faculdade de Matem\'atica
\newline\indent
Instituto de Ci\^{e}ncias Exatas e Naturais
\newline\indent
Universidade Federal do Par\'a
\newline\indent
Avenida Augusto corr\^{e}a 01, 66075-110, Bel\'em, PA, Brazil}
\email{\href{mailto: joaojunior@ufpa.br}{joaojunior@ufpa.br}}

\thanks{C. O. Alves was partially supported by CNPq/Brazil 307045/2021-8 and Projeto Universal FAPESQ-PB 3031/2021. J. R. Santos J\'unior is the corresponding author and he was partially supported by CNPq 313766/2021-5, Brazil.}


\begin{document}

\begin{abstract}
In this work, we establish a new method to find critical points of differentiable functionals defined in Banach spaces which belong to a suitable class ($\mathcal{J}$) of functionals. Once given a functional $J$ in the class ($\mathcal{J}$), the central idea of the referred method consists in defining a real function $\zeta$ of a real variable, called {\it energy function}, which is naturally associated to $J$ in the sense that the existence of real critical points for $\zeta$ guarantees the existence of critical points for the functional $J$. As a consequence, we are able to solve some variational elliptic problems, whose associated energy functional belongs to ($\mathcal{J}$) and provide a version of the mountain pass theorem for functionals in the class ($\mathcal{J}$) that allows us to obtain mountain pass solutions without the so-called Ambrosetti-Rabinowitz condition. \end{abstract}

\maketitle

\begin{center}
\begin{minipage}{12cm}
\tableofcontents
\end{minipage}
\end{center}

\bigskip


\section{Introduction}\label{se:intro}

Over the years, many methods have been developed to attack different classes of partial differential equations. Even though some methods are renowned for their range of applications, it is of course expected that each one of them has its advantages and disadvantages. In fact, the experience shows us that depending on the level of complexity or specificity of a partial differential equation certain approaches are more appropriated and effective than others. In this context, it is crucial to have a diversified theoretical framework which allows that different techniques complement each other in order to provide a more complete understanding of the problems. 

\medskip

In this work, we have established an alternative method that allows us to complement some classical results of the literature, involving differential equations. In fact, this referred method, which we will call {\it Method of the Energy Function} (MEF), has proved to be effective to treat some relevant classes of variational elliptic partial differential problems, for which, some progress is made in the present article, see Section \ref{se:app}. For instance, by using the energy function approach, we obtain a version of the Mountain Pass Theorem suitable for problems that do not require the Ambrosetti-Rabinowitz growth condition, see Theorem \ref{TPM}.  

\medskip

In order to give a more clear idea about the subject we are going to treat, let us consider a functional $J:E\rightarrow\R$ in the form
$J = \Psi - \Phi$, where $E$ is a reflexive Banach space and $\Psi,\Phi\in C^1(E,\R)$. When $J$ is of class $(\mathcal{J})$ in $E$, that is, $\Phi$ and $\Psi$ satisfy certain conditions to be stated in the next section, it is possible to relate $J$ with a real $C^1$-function $\zeta:[0,\infty)\rightarrow\R$, which we will call the {\it energy function}.

\medskip

The main feature on the approach we are about to describe is that it relates the existence of critical points for $\zeta$ with the existence of critical points for the functional $J$. More precisely, under appropriated conditions, we can ``change'' the effort of searching for critical points of functionals defined in infinite-dimensional Banach spaces for the task of finding critical points of a one variable real function, which is certainly more treatable, since it does not require more than basic tools of the classical differential Calculus to be executed. 

\medskip

The idea of trying to relate the energy functional of a certain elliptic problem to a real function was inspired by the work of Arcoya, Santos J\'unior and Su\'arez \cite{ASS}, which studied a problem involving a degenerate Kirchhoff operator of the form
\begin{equation}\label{0.2}
\left\{ \begin{array}{lll}
-m(\Vert u\Vert)\Delta u = f(u) &\mbox{in}&\Omega, \\
\hspace{2.0cm}u = 0 &\mbox{on}& \partial\Omega, \end{array} \right.
\end{equation}
where $\Omega\subset\R^N$ ($N\geq 1$) is a smooth bounded domain, $m:[0,\infty)\rightarrow\R$ and $f:\R\rightarrow\R$ are continuous functions satisfying appropriated conditions. The authors in \cite{ASS} defined a function $\alpha:[0,\infty)\rightarrow\R$ given by
\begin{equation}\label{funct}
\alpha(r) = \max_{\Vert u\Vert^2\leq r \atop u\in H_{0}^{1}(\Omega)}\int_{\Omega}F^{\ast}(u)dx, 
\end{equation}
where
$$ F^{\ast}(t)=\int_0^t f^{\ast}(s)ds,\,\,\mbox{with}\,\,
f^{\ast}(t) = 
\left\{ \begin{array}{lll}
f(0) &\mbox{if}& t < 0, \\
f(t) &\mbox{if}& 0 \leq t < s_{\ast}, \\
0 &\mbox{if}& s_{\ast} < t, \end{array} \right. $$
and showed that $\alpha$ is differentiable in $(0,\infty)$ with
$$ \alpha'(r) = \frac{1}{2r}\max_{u\in \mathcal{S}_r} \int_{\Omega}f^{\ast}(u)udx, $$
where
$$ \mathcal{S}_r := \left\lbrace u\in H_0^1(\Omega); \Vert u\Vert^2 \leq r \,\,\mbox{and}
\,\,\int_{\Omega}F^{\ast}(u)dx = \alpha(r) \right\rbrace. $$

\medskip

It is important to point out that it is not obvious or intuitive that $\alpha$, defined as in \eqref{funct}, is a differentiable, or even continuous, map regarding to $r$, mainly because the function which attains the maximum in \eqref{funct} cannot be unique. In fact, in the particular situation considered in \cite{ASS}, in order to prove the differentiability of $\alpha$, the authors were inspired by arguments first used in Szulkin and Weth \cite{SW}. Another important detail on \cite{ASS} that should be highlighted is that, the map $\alpha$ is only used to show that the weak solution found to the problem \eqref{0.2} is nontrivial and no relation was established between critical points of $\alpha$, or even of any other real function involving $\alpha$, and critical points of the energy functional associated to the problem \eqref{0.2}. 

\medskip

\medskip

 The paper is divided as follows: In Section \ref{se:method}, we introduce the energy function $\zeta:[0,\infty)\rightarrow\R$ (see (\ref{zeta})) in an abstract context and give general conditions under which $\zeta$ is $C^1((0,\infty), \R)$ and its critical points have a strong connexion with a certain class of critical points of $J$, see Theorem \ref{TME}. Furthermore, we provide an application of the method (MEF) to a class of semilinear problems which complements results obtained in \cite{MWW}. In Section \ref{se:mountain}, we apply the method (MEF) to get a new version of the Mountain Pass Theorem, see Theorem \ref{TPM}. Some interesting facts about Theorem \ref{TPM} are the following: 1) It provides a new characterization of the mountain pass level as a ``maxmin'' type constant. In fact, at least in our best knowledge, such a characterization has not previously appeared in the literature; 2) It allows to treat problems involving nonlinearities satisfying conditions weaker than Ambrosetti-Rabinowitz condition \cite{AR}, which have not been required in papers that address the matter, see for instance \cite{ASM, GJF, HMR, GS, LLu, LW, LW2, Liu, LiW, MP, MS} and references there in. Finally, in Section \ref{se:app}, in order to illustrate the applicability of Theorem \ref{TPM}, we revisite the studies made in Willem and Zou \cite{W}, Schechter and Zou \cite{SZ}, Miyagaki and Souto \cite{MS}, Azzollini and Pomponio \cite{AP} and Barile and Figueiredo \cite{BF}.


\section{The Method (MEF)}\label{se:method}

Throughout this article, $E$ denotes a reflexive Banach space. We say that a functional $J$ belongs to class $(\mathcal{J})$ in $E$, if $J:E\rightarrow\R$, $J = \Psi - \Phi$ and $\Psi, \Phi\in C^1(E,\R)$ satisfy the following hypotheses:
\begin{itemize}
\item[$(\Psi_1)$] $\Psi$ is convex, coercive and  weakly lower semicontinuous;  
\item[$(\Psi_2)$] $\tilde{\mathcal{A}}_r =\lbrace u\in E;\,\Psi(u) \leq r\rbrace$ for all $r \geq 0$, $\tilde{\mathcal{A}}_{0}=\{0\}$ and $0 \in \mbox{int}(\tilde{\mathcal{A}}_r)$ for each $r> 0$;
\item[$(\Psi_3)$] $\Psi'(u)u\neq 0$ for all $u\in\partial\tilde{\mathcal{A}}_r$, where $\partial\tilde{\mathcal{A}}_r =\lbrace u\in E;\,\Psi(u) = r\rbrace $;
\item[$(\Psi_4)$] For each $u\not= 0$ and $r>0$, there exists a unique $t_u(r) > 0$ such that $t_u(r)u\in\partial\tilde{\mathcal{A}}_r$;
\item[$(\Psi_5)$] If $u_n \rightharpoonup u_0$ in $E$ and $\Psi(u_n) \to \Psi(u_0)$, then $u_n \to u_0$ in $E$;
\item[$(\Phi_1)$]$\Phi$ and  $u\mapsto\Phi'(u)u/\Psi'(u)u$ are weakly upper semicontinuous;
\item[$(\Phi_2)$] If $\Phi'(u) = 0$, then $\Phi(u)\leq 0$;
\item[$(\Phi_3)$] There exists a sequence $\lbrace u_n\rbrace\subset E$ such that $u_n\rightarrow 0$ and $\Phi(u_n) > 0$, for all $n\in\N$.
\end{itemize}

\medskip

Some consequences can be concluded from previous assumptions: 

\medskip

\begin{itemize}
\item[$(1)$] The coercivity in $(\Psi_1)$ ensures that the boundaries $\partial\tilde{\mathcal{A}}_r$ are nonempty, and it also implies the boundedness of the sets $\tilde{\mathcal{A}}_r$. 

\medskip

\item[$(2)$] From $(\Psi_2)$, it follows that $\Psi(u)=0$ if, and only if, $u=0$ and, moreover, $\Psi(u)>0$ if $u\neq 0$. 

\medskip

\item[$(3)$]  Combining information in (2), the coercivity in $(\Psi_1)$ and $(\Psi_3)$, we conclude that $\Psi'(u)u> 0$ for all $u\neq 0$. 
\end{itemize}

\medskip

\begin{proposition}\label{pro1.1}
Suppose that $\Psi$ satisfies $(\Psi_1)$, $0 \in \mbox{int}(\tilde{\mathcal{A}}_r)$ for each $r>0$, and $\Phi$ is weakly upper semicontinuous and satisfies $(\Phi_2)-(\Phi_3)$. Then, for each $r > 0$, $\Phi$ is bounded from above on $\tilde{\mathcal{A}}_r$ and its maximum is attained at the boundary. In other words, there exists $u_r\in\partial\tilde{\mathcal{A}}_r$ such that
$$ \Phi(u_r) = \max_{u\in\tilde{\mathcal{A}}_r}\Phi(u)
= \max_{u\in\partial\tilde{\mathcal{A}}_r}\Phi(u). $$
\end{proposition}

\begin{proof}

Since $\Psi$ satisfies $(\Psi_1)$, we have that $\tilde{\mathcal{A}}_r$ is convex, closed and bounded, and so, $\tilde{\mathcal{A}}_r$ is weakly closed and bounded, from where it follows that $\tilde{\mathcal{A}}_r$ is weakly compact. On the other hand, as $\Phi$ is weakly upper semicontinuous, it is bounded from above on $\tilde{\mathcal{A}}_r$ and there exists $u_r\in \tilde{\mathcal{A}}_r$ such that
$$ \Phi(u_r) = \max_{u\in\tilde{\mathcal{A}}_r}\Phi(u). $$

Now, we are going to show that $u_r\in \partial \tilde{\mathcal{A}}_r$. Arguing by contradiction, assume that $u_r$ belongs to the interior of $\tilde{\mathcal{A}}_r$, which is open, because $\Psi$ is continuous and
$$
\mathcal{A}_r:=\mbox{int}(\tilde{\mathcal{A}}_r) = \lbrace u\in E;\,\Psi(u) <  r\rbrace. 
$$
In this case, $u_r$ is a critical point of $\Phi$, that is, $\Phi'(u_r) = 0$. Using ($\Phi_3$) together with the fact that $0 \in \mbox{int}(\tilde{\mathcal{A}}_r)$ for each $r>0$ , there exists $n_0 \in\N$ such that $u_{n_0}\in\tilde{\mathcal{A}}_r$ and
$$ 
0 < \Phi(u_{n_0}) \leq \max_{u\in\tilde{\mathcal{A}}_r}\Phi(u) = \Phi(u_r). 
$$
However, from ($\Phi_2$), we also have that $\Phi(u_r)\leq0$, which is a contradiction, proving that $u_r\in\partial\mathcal{A}_r$.

\end{proof}

\begin{remark}\label{obs}
Let us consider the set
$$ \mathcal{G}_r := \left\lbrace u\in\tilde{\mathcal{A}_r};
\hspace{0.1cm}\Phi(u) = \max_{v\in\tilde{\mathcal{A}_r}}\Phi(v)\right\rbrace. $$
It is an immediate consequence of the proof of Proposition \ref{pro1.1} that, if $\Psi$ satisfies $(\Psi_1)$ and $\Phi$ is weakly upper semicontinuous, then the set $\mathcal{G}_r$ is nonempty. So, it makes sense to consider the hypothesis
\begin{itemize}
\item[$(\Phi_2')$] For each $r> 0$, there exists $c_r>0$, such that $c_r\Phi(u)\leq \Phi'(u)u$ for all $u\in \mathcal{G}_r$.
\end{itemize}
Note that, Proposition \ref{pro1.1} (as well as all subsequent results that make use of it) still holds true if we replace hypothesis $(\Phi_2)$ by $(\Phi_2')$. It is important to point out that $(\Phi_2)$ and $(\Phi_{2}')$ does not imply each other.
\end{remark}

\medskip

From previous remark, the set $\mathcal{G}_r$ is well-defined. Moreover, by Proposition \ref{pro1.1}, one has $\mathcal{G}_r \subset \partial\tilde{\mathcal{A}_r}$.

\medskip

In the following, under the assumptions of Proposition \ref{pro1.1}, it is possible to introduce a real function of real variable which is naturally related to the functional $J$, in a sense that we will clarify later on (see Theorem \ref{TME}). This function plays a crucial role in the next results to be presented in this manuscript.

\medskip

Suppose that the assumptions in Proposition \ref{pro1.1} hold. We call \textit{energy function} to the function $\zeta:[0,\infty)\rightarrow\R$ defined by
\begin{equation} \label{zeta} 
\zeta(r) := r - \varphi(r),
\end{equation}
where $\varphi(r) := \displaystyle \max_{u\in\partial\tilde{\mathcal{A}}_r} \Phi(u)$ if $r>0$ and $\varphi(0)=\Phi(0)$.

\medskip

Observe that the energy function and the functional $J$ are related by the following equality
\begin{equation} \label{energyfunction}
\zeta(r) =  r - \max_{u\in\partial\tilde{\mathcal{A}_r}}\Phi(u)
= \min_{u\in\partial\tilde{\mathcal{A}}_r}J(u)=J(u_r) \ \mbox{for all $r>0$ and $\zeta(0)=J(0)$}.
\end{equation}

In what follows, our main goal is to show that we can change the task of finding critical points of $J$ for the simpler task of finding critical points for $\zeta$. In order to establish this relation, we will study the differentiability of function $\varphi$.

\begin{proposition}
Suppose that $\Psi$ satisfies $(\Psi_1)$, $(\Psi_2)$ and $(\Psi_4)$, and $\Phi$ is weakly upper semicontinuous and satisfies $(\Phi_2)-(\Phi_3)$. Then, function $\varphi:[0,\infty)\to\R$ is continuous.
\end{proposition}

\begin{proof}

It is enough to take an arbitrary sequence $\lbrace r_n\rbrace\subset (0, \infty)$ with $r_n \rightarrow r_0\geq 0$. Before that, observe that for each $n\in\N$, we know from Proposition \ref{pro1.1} that, there exists a function $u_n\in\partial\tilde{\mathcal{A}_{r_{n}}}$ such that
$$ \Phi(u_n) = \max_{u\in\partial\tilde{\mathcal{A}}_{r_n}}\Phi(u) = \max_{u\in\tilde{\mathcal{A}}_{r_n}}\Phi(u). $$

We will show that $\lbrace u_n\rbrace$ is bounded. In fact, since $\lbrace r_n\rbrace$ is bounded, there exists $R > 0$ such that $\Psi(u_n) = r_n \leq R$ for all $n\in\N$, showing that $\{u_{n}\}\subset \tilde{\mathcal{A}}_R$, which is a bounded set by $(\Psi_1)$. Hence, up to a subsequence,
\begin{equation}\label{fraca}
u_n \rightharpoonup u_0 \,\,\mbox{in}\,\,E. 
\end{equation}
Next, we will consider two cases, namely $r_0=0$ and $r_0>0$. \\

\noindent {\bf Case 1:} $r_{0}=0$.

\medskip

In this case, by $(\Psi_{2})$, the lower weak semicontinuity of $\Psi$ and \eqref{fraca}, one has
$$
0\leq \Psi(u_{0})\leq  \liminf_{n\rightarrow\infty}\Psi(u_{n})= \liminf_{n\rightarrow\infty} r_{n}=0.
$$
Therefore, $\Psi(u_{0})=0$ and, again, by $(\Psi_{2})$, $u_{0}=0$.

\medskip

Finally, from $(\Phi_{3})$, there exists $\{v_{n}\}\subset E$ with $v_{n}\to 0$ and $\Phi(v_{n})>0$. Now, from $(\Psi_2)$, if necessary, we can renumerate $\{v_{n}\}$ in such a way that $v_{n}\in \tilde{\mathcal{A}}_{r_{n}}$. Thus, since $\Phi$ is continuous and weakly upper semicontinuous, we get
$$
\Phi(0)= \lim_{n\to\infty}\Phi(v_{n})\leq  \liminf_{n\rightarrow\infty}\Phi(u_{n})\leq \limsup_{n\rightarrow\infty}\Phi(u_{n})\leq \Phi(0),
$$
showing that  
$$
\varphi(r_{n})=\Phi(u_{n})\to \Phi(0)=\varphi(0).
$$

\noindent {\bf Case 2:} $r_{0}>0$.

\medskip

Since $\Psi$ is weakly lower semicontinuous,
$$ \Psi(u_0) \leq \liminf_{n\rightarrow\infty} \Psi(u_n)
= \liminf_{n\rightarrow\infty} r_n = r_0, $$
which means that $u_0\in\tilde{\mathcal{A}}_{r_0}$. Thus,
\begin{equation}\label{varphicont1}
\Phi(u_0) \leq \varphi(r_0).
\end{equation}
On the other hand, consider $u_{\ast}\in\partial\tilde{\mathcal{A}}_{r_0}$ such that
$$
\varphi(r_0) = \Phi(u_{\ast}). 
$$
Gathering $(\Psi_4)$ and coercivity of $\Psi$, there exists a bounded sequence $\{t_n\} \subset (0, +\infty)$ such that $t_n u_{\ast}\in\partial\tilde{\mathcal{A}}_{r_n}$. Then
$$
\Phi\left(t_n u_{\ast}\right) \leq \Phi(u_n), \quad \forall n \in \mathbb{N}.  
$$
\begin{claim} $t_n\rightarrow 1$ as $n \to \infty$.
	
\end{claim} 
\noindent Assume that for some subsequence of $\{t_n\}$, still denoted by itself, we have $t_n \to t_0$. Since $\Psi$ is continuous, we must have 
$$
\Psi(t_0 u_{\ast})=\lim_{n\rightarrow\infty}\Psi(t_n u_{\ast}) = \lim_{n\rightarrow\infty} r_n = r_0 = \Psi(u_{\ast}),
$$
from where it follows that $t_0 u_{\ast} \in \partial\tilde{\mathcal{A}}_{r_0}$. As $u_{\ast}\in\partial\tilde{\mathcal{A}}_{r_0}$, $(\Psi_4)$ yields that $t_0=1$ and the Claim is proved. 

\medskip

The claim above permits to conclude that 
$$
\varphi(r_0) =\Phi\left(u_{\ast}\right)=\limsup_{n\rightarrow\infty}\Phi\left(t_n u_{\ast}\right)
\leq \limsup_{n\rightarrow\infty}\Phi(u_n). 
$$
From $(\Phi_1)$,
\begin{equation}\label{varphicont2}
\varphi(r_0) \leq \limsup_{n\rightarrow\infty}\Phi(u_n) \leq \Phi(u_0).
\end{equation}
Employing (\ref{varphicont1}) and (\ref{varphicont2}), one finds
$$
\varphi(r_0) = \Phi(u_0). 
$$
Again, by $(\Psi_4)$ and coercivity of $\Psi$, there exists a bounded sequence $\{s_n\} \subset (0,+\infty) $ such that $s_n u_0\in\partial\tilde{\mathcal{A}}_{r_n}$ and $s_n \to 1$. Hence, 
$$
\Phi(u_0) = \liminf_{n\rightarrow\infty}\Phi(s_n u_0) \leq \liminf_{n\rightarrow\infty}\Phi(u_n). 
$$
On the other hand, as $\Phi$ is weakly upper semicontinuous,
$$
\limsup_{n\rightarrow\infty}\Phi(u_n) \leq \Phi(u_0). 
$$
The last two inequalities give 
$$
\lim_{n\rightarrow\infty}\Phi(u_n) = \Phi(u_0). 
$$
Therefore,
$$
\lim_{n\rightarrow\infty} \varphi(r_n) = \lim_{n\rightarrow\infty} \Phi(u_n)
= \Phi(u_0) = \varphi(r_0). 
$$

\end{proof}

\begin{proposition}\label{pro1.5}
Suppose that $\Psi$ and $\Phi$ satisfy $(\Psi_1)-(\Psi_3)$ and $(\Phi_1)-(\Phi_3)$ respectively. Then, the functional $u\mapsto \frac{\Phi'(u)u}{\Psi'(u)u}$ attains the maximum in the set
$$
\mathcal{G}_r = \left\lbrace u\in \tilde{\mathcal{A}_r}; \,\Phi(u) = \varphi(r)\right\rbrace. 
$$
\end{proposition}

\begin{proof}

By $(\Psi_1)$, $\mathcal{G}_{r}$ is a bounded set. Now, we will prove that $\mathcal{G}_{r}$ is weakly closed. Since $\tilde{\mathcal{A}_r} \subset \overline{B}_R(0) $ for some $R>0$,  $\overline{B}_R(0) $ is weakly closed and $\overline{B}_R(0) $ is metrizable by the weak topology, it is enough to prove that if $\lbrace u_n\rbrace\subset \mathcal{G}_{r}$ and $u_n\rightharpoonup u_0$ in $E$, then $u_0 \in \mathcal{G}_{r}$. 
From $(\Psi_1)$,
\begin{equation}\label{acaba}
\Psi(u_0) \leq \liminf_{n\rightarrow\infty}\Psi(u_n) = r, 
\end{equation}
which means that $u_{0}\in \tilde{\mathcal{A}}_{r}$ and
$$
\Phi(u_0) \leq \varphi(r). 
$$
On the other hand, since $\Phi$ is weakly upper semicontinuous and $\lbrace u_n\rbrace\subset \mathcal{G}_{r}$, it follows that
\begin{equation}\label{logo}
\varphi(r) = \limsup_{n\rightarrow\infty}\Phi(u_n) \leq \Phi(u_0). 
\end{equation}
So, from \eqref{acaba}, \eqref{logo} and Proposition \ref{pro1.1}, we get $u_0\in \mathcal{G}_{r}$. Therefore, $\mathcal{G}_{r}$ is weakly closed and bounded and, consequently, weakly compact. Furthermore, by $(\Psi_3)$ and $(\Phi_1)$, the function $u\mapsto\Phi'(u)u/\Psi'(u)u$ is well defined and is weakly upper semicontinuous, then the maximum of this functional is attained in $\mathcal{G}_r$.
\end{proof}

We say that a function $u_{r}\in \mathcal{G}_r$ is an {\it energy maximum type point} if 
$$
\frac{\Phi'(u_{r})u_{r}}{\Psi'(u_{r})u_{r}}=\max_{u\in \mathcal{G}_r}\frac{\Phi'(u)u}{\Psi'(u)u}.
$$
Observe that an energy maximum type point $u_{r}\in \partial\tilde{\mathcal{A}}_{r}$ maximizes at the same time $\Phi$ in $\tilde{\mathcal{A}}_{r}$ and $u\mapsto \frac{\Phi'(u)u}{\Psi'(u)u}$ in $\mathcal{G}_r$, and, for sure, it is nontrivial when $r>0$.

\medskip

Before proving the differentiability of function $\varphi$, we need the following lemma about the differentiability of mapping $r\mapsto t_{u}(r)$, for each $u\neq 0$ fixed.

\medskip

\begin{lemma}\label{derivamos}
Suppose that $\Psi$ is coercive and satisfies $(\Psi_2)-(\Psi_4)$. Then, for each $u\neq 0$ fixed, the application $r\mapsto t_{u}(r)$, given in $(\Psi_4)$, is $C^{1}((0, \infty), \R)$ with
\begin{equation}
t_{u}'(r)=\frac{1}{\Psi'(t_{u}(r)u)u}, \ \mbox{for all $r>0$}.
\end{equation}
\end{lemma}

\begin{proof}
In what follows, we define $h:[0,\infty)\to\R$ by $h_{u}(t)=\Psi(tu)$. From the coercivity of $\Psi$ and $(\Psi_2)-(\Psi_3)$, we infer that $h_{u}(0)=0$, $\displaystyle \lim_{t\to\infty}h_{u}(t)=\infty$ and $h_{u}$ is increasing. Thus, there exists the inverse $h_{u}^{-1}$ and from $(\Psi_{4})$
$$
t_{u}(r)=h_{u}^{-1}(r), \ \mbox{for all $r>0$}.
$$
The result follows now from the chain's rule.
\end{proof}

\medskip

After that previous analysis, we are ready to prove that $\varphi$ is $C^{1}$ in $(0, \infty)$.

\medskip

\begin{proposition}\label{provardif}
Suppose that $J=\Psi-\Phi$ belongs to class $(\mathcal{J})$ in $E$. Then, $\varphi\in C^{1}((0,\infty), \R)$ and
$$
\varphi'(r) = \max_{u\in\mathcal{G}_r}\frac{\Phi'(u)u}{\Psi'(u)u}.
$$
\end{proposition}

\begin{proof}

Fix $r_0 > 0$ and take $u_{r_0}\in\mathcal{G}_{r_0}$ and $u_{r_0 + r}\in\mathcal{G}_{r_0 + r}$, with $r>0$. Thereby, 
$$ 
\varphi(r_0+r)-\varphi(r_0) = 
\Phi(u_{r_0+r})-\Phi(u_{r_0}). 
$$
By $(\Psi_4)$, there exists a unique $t_{u_{r_0}}(r_0 + r)$ such that $t_{u_{r_0}}(r_0 + r)u_{r_0}\in\partial\tilde{\mathcal{A}}_{r_0+r}$. So,
$$
\Phi\left(t_{u_{r_0}}(r_0 + r)u_{r_0}\right)-\Phi(u_{r_0}) \leq
\varphi(r_0+r)-\varphi(r_0). 
$$
Since $t_{u_{r_0}}(r_0) = 1$, it follows from Lemma \ref{derivamos} and Mean Value Theorem that there exists a constant $0 < c_r < 1$ such that
$$
\Phi\left(t_{u_{r_0}}(r_0 + r)u_{r_0}\right)-\Phi(u_{r_0})=r\frac{\Phi'\left(t_{u_{r_0}}(r_0 + c_r r)u_{r_0}\right)u_{r_0}}{\Psi'(t_{u_{r_0}}(r_0+c_r r)u_{r_0})u_{r_0}}. 
$$
Consequently,
\begin{equation}\label{ineq1}
r\frac{\Phi'\left(t_{u_{r_0}}(r_0 + c_r r)u_{r_0}\right)u_{r_0}}{\Psi'(t_{u_{r_0}}(r_0+c_r r)u_{r_0})u_{r_0}}\leq \varphi(r_0+r)-\varphi(r_0). 
\end{equation}

On the other hand, there exists a unique $t_{u_{r_0+r}}(r_0)$ such that $t_{u_{r_0+r}}(r_0)u_{r_0+r}\in\partial\tilde{\mathcal{A}}_{r_0}$. Then
$$ 
\varphi(r_0 + r)-\varphi(r_0) \leq 
\Phi(u_{r_0+r}) - \Phi\left(t_{u_{r_0+r}}(r_0)u_{r_0+r}\right). 
$$
Again, as $t_{u_{r_0+r}}(r_0+r) = 1$, it follows from Lemma \ref{derivamos} and Mean Value Theorem that there exists a constant $0 < d_r < 1$ such that
\begin{equation}\label{ineq2}
\varphi(r_0 + r)-\varphi(r_0) \leq
r\frac{\Phi'\left(t_{u_{r_0+r}}(r_0 + d_r r)u_{r_0+r}\right)u_{r_0+r}}
{\Psi'\left(t_{u_{r_0+r}}(r_0 + d_r r)u_{r_0+ r}\right)u_{r_0+r}}.
\end{equation}
From \eqref{ineq1} and \eqref{ineq2},
\begin{equation}\label{agora}
r\frac{\Phi'\left(t_{u_{r_0}}(r_0 + c_r r)u_{r_0}\right)u_{r_0}}{\Psi'(t_{u_{r_0}}(r_0+c_r r)u_{r_0})u_{r_0}}\leq \varphi(r_0 + r)-\varphi(r_0) \leq
r\frac{\Phi'\left(t_{u_{r_0+r}}(r_0 + d_r r)u_{r_0+r}\right)u_{r_0+r}}
{\Psi'\left(t_{u_{r_0+r}}(r_0 + d_r r)u_{r_0+ r}\right)u_{r_0+r}}. 
\end{equation}

\medskip

Since $\Psi, \Phi$ are $C^{1}$ functionals, $r\mapsto t_{u_{r_{0}}}(r)$ is continuous (see Lemma \ref{derivamos}) and $t_{u_{r_0}}(r_{0})=1$, it is clear that
$$
\frac{\Phi'\left(t_{u_{r_0}}(r_0 + c_r r)u_{r_0}\right)u_{r_0}}{\Psi'(t_{u_{r_0}}(r_0+c_r r)u_{r_0})u_{r_0}}\to \frac{\Phi'(u_{r_0})u_{r_0}}{\Psi'(u_{r_0})u_{r_0}}, \quad \mbox{as} \quad r\to 0. 
$$
In particular, if we choose $u_{r_0}\in\mathcal{G}_{r_0}$ as a maximum energy type point, we have 
\begin{equation}\label{agora2}
\frac{\Phi'\left(t_{u_{r_0}}(r_0 + c_r r)u_{r_0}\right)u_{r_0}}{\Psi'(t_{u_{r_0}}(r_0+c_r r)u_{r_0})u_{r_0}}\to  \max_{u\in\mathcal{G}_{r_{0}}}\frac{\Phi'(u)u}{\Psi'(u)u}, \quad \mbox{as} \quad r\to 0. 
\end{equation}

\medskip

The behaviour of the right side term in inequality \eqref{agora}, as $r$ tends to zero, is much more delicate to be analyzed, since $u_{r_{0}+r}$ is depending on $r$. To do that, we need to prove some previous claims.

\medskip

\noindent {\bf Claim 1:} For each sequence $r_n \to 0^+$, we have that $u_{r_0+ r_n}\rightharpoonup u_{\ast}$, for some maximum energy type point $u_{\ast}\in \mathcal{G}_{r_{0}}$.

\medskip

From $(\Psi_4)$ and Proposition \ref{pro1.1}, for each $u\in \partial\tilde{\mathcal{A}}_{r_{0}}$, one has $t_{u}(r_{0}+r_n)u\in \partial\tilde{\mathcal{A}}_{r_{0}+r_n}$ with
\begin{equation}\label{ineq3}
\Phi(t_{u}(r_{0}+r_n)u)\leq \Phi(u_{r_0+ r_n})
\end{equation} 
Moreover, by $(\Psi_1)$, the sequence $\{u_{r_{0}+r_n}\}$ is bounded, and so, up to a subsequence, there exists $u_{\ast}\in E$ such that
\begin{equation}\label{conv1}
u_{r_0+ r_n}\rightharpoonup u_{\ast}.
\end{equation}
From Lemma \ref{derivamos}, $(\Phi_1)$, \eqref{ineq3} and \eqref{conv1}, 
$$
\Phi(t_{u}(r_{0})u)\leq \Phi(u_{\ast}).
$$
Since $u\in \partial\tilde{\mathcal{A}}_{r_{0}}$, it follows from $(\Psi_4)$ that $t_{u}(r_{0})=1$. Thus
\begin{equation}\label{ineq4}
\Phi(u)\leq \Phi(u_{\ast}), \ \mbox{for all $u\in \partial\tilde{\mathcal{A}}_{r_{0}}$}.
\end{equation}
On the other hand, from $\Psi(u_{r_{0}+r_n})=r_{0}+r_n$, \eqref{conv1} and $(\Psi_1)$, 
\begin{equation}\label{ineq5}
\Psi(u_{\ast})\leq r_{0}.
\end{equation}
It is a consequence of \eqref{ineq4}, \eqref{ineq5} and Proposition \ref{pro1.1} that $u_{\ast}\in\partial\tilde{\mathcal{A}}_{r_{0}}$ (in particular $u_{\ast}$ is nontrivial) and 
$$
\Phi(u_{\ast})=\max_{u\in \tilde{\mathcal{A}}_{r_{0}}}\Phi(u),
$$
and so,  
\begin{equation}\label{Alves}
u_{\ast}\in \mathcal{G}_{r_{0}}. 
\end{equation}
Since $u\mapsto \Phi'(u)u/\Psi'(u)u$ is also weakly upper semicontinuous, by replacing $\Phi$ for $u\mapsto \Phi'(u)u/\Psi'(u)u$ and repeating the same previous argument, we conclude that
$$
\frac{\Phi'(u)u}{\Psi'(u)u}\leq \frac{\Phi'(u_{\ast})u_{\ast}}{\Psi'(u_{\ast})u_{\ast}}, \ \mbox{for all $u\in \mathcal{G}_{r_{0}}$}.
$$
The proof of the claim follows now from \eqref{conv1}, \eqref{Alves} and last inequality.

\medskip

\noindent {\bf Claim 2:} $t_{u_{r_0+r_{n}}}(r_0 + d_{r_{n}} r_{n})\to 1$ as $n\to \infty$.

\medskip

From $(\Psi_4)$,
\begin{equation}\label{equ1}
\Psi(t_{u_{r_0+r_{n}}}(r_0 + d_{r_{n}} r_{n})u_{r_{0}+r_{n}})=r_{0}+d_{r_{n}}r_{n}.
\end{equation}
Moreover, by $(\Psi_1)$, $\{u_{r_{0}+r_{n}}\}$ is bounded for $n$ large enough. Since $\Psi(u_{r_{0}+r_{n}})=r_{0}+r_{n}$, $\{u_{r_{0}+r_{n}}\}$ is far away from the origin of $E$ and, consequently, $\{t_{u_{r_0+r_{n}}}(r_0 + d_{r_{n}} r_{n})\}$ is bounded in $\R$. Therefore, up to a subsequence,
\begin{equation}\label{conv2}
t_{u_{r_0+r_{n}}}(r_0 + d_{r_{n}} r_{n})\to t_{\ast}
\end{equation}
for some $t_{\ast}\geq 0$. From $(\Psi_1)$, $(\Psi_2)$, Claim 1, \eqref{equ1} and \eqref{conv2},
$$
\Psi(t_{\ast} u_{\ast})\leq \lim_{n\rightarrow \infty}\Psi(t_{u_{r_0+r_{n}}}(r_0 + d_{r_{n}} r_{n})u_{r_{0}+r_{n}})=\lim_{n\rightarrow \infty}(r_{0}+d_{r_{n}}r_{n})=r_{0}=\Psi(u_{\ast}).
$$
Showing that $t_{\ast} u_{\ast}\in \tilde{\mathcal{A}}_{r_{0}}$. On the other hand, since $h_{u_{\ast}}(t)=\Psi(tu_{\ast})$ is increasing (see proof of Lemma \ref{derivamos}), we conclude that $t_{\ast}\leq 1$. Next, we are going to prove that $(\Psi_5)$ implies $t_{\ast}=1$. In fact, it is a consequence of Claim 1, $(\Psi_5)$ and
$$
\Psi(u_{r_{0}+r_{n}})=r_{0}+r_{n}\to r_{0}=\Psi(u_{\ast}),
$$
that
\begin{equation}\label{conv3}
u_{r_{0}+r}\to u_{\ast} \ \mbox{in $E$}.
\end{equation}
From \eqref{equ1}, \eqref{conv2}, \eqref{conv3} and the continuity of $\Psi$, we get
$$
\Psi(t_{\ast}u_{\ast})=r_{0}.
$$
Since $u_{\ast}\in  \partial\tilde{\mathcal{A}}_{r_{0}}$, it follows from $(\Psi_4)$ that $t_{\ast}=1$, and the claim is proved.

\medskip

Hence, from Claims 1 and 2, 
\begin{equation}\label{agora3}
\limsup_{n \to \infty}\frac{\Phi'\left(t_{u_{r_0+r_{n}}}(r_0 + d_{r_{n}} r_{n})u_{r_0+r_{n}}\right)u_{r_0+r_{n}}}
{\Psi'\left(t_{u_{r_0+r_{n}}}(r_0 + d_{r_{n}} r_{n})u_{r_0+ r_{n}}\right)u_{r_0+r_{n}}} \leq \frac{\Phi'(u_{\ast})u_{\ast}}{\Psi'(u_{\ast})u_{\ast}}=\max_{u\in\mathcal{G}_{r_{0}}}\frac{\Phi'(u)u}{\Psi'(u)u}.
\end{equation}
Since $r_{n}\to 0^{+}$ is an arbitrary sequence, it follows from \eqref{agora}, \eqref{agora2} and \eqref{agora3} that $\varphi$ is differentiable, and
$$
\varphi'(r_{0})=\max_{u\in\mathcal{G}_{r_{0}}}\frac{\Phi'(u)u}{\Psi'(u)u}.
$$
The continuity of $\varphi'$ is also a straightforward consequence of $(\Psi_5)$.
\end{proof}

A careful analysis about the proof of Proposition \ref{provardif} shows us that if $(\Psi_5)$ is not required, we can still guarantee the existence of the right hand derivative of $\varphi$, more specifically, we can proof that
$$
\varphi'_{+}(r_{0})=\max_{u\in\mathcal{G}_{r_{0}}}\frac{\Phi'(u)u}{\Psi'(u)u}, \ \forall \ r_{0}>0,
$$ 
and 
$$
\limsup_{r\to 0^{-}}\frac{\varphi(r_{0}+r)-\varphi(r_{0})}{r}\leq \max_{u\in\mathcal{G}_{r_{0}}}\frac{\Phi'(u)u}{\Psi'(u)u}, \ \forall \ r_{0}>0.
$$

\begin{theorem}\label{TME}
Let $E$ be a reflexive Banach space and $J$ be a functional of class $(\mathcal{J})$ in $E$. A maximum energy type point $u_r\in \mathcal{G}_r$ is a critical point of the functional $J$, for some $r>0$ if, and only if, $r$ is a critical point of the energy function $\zeta$.
\end{theorem}
\begin{proof}

Suppose that $u_r$ is a maximum energy type point. By Proposition \ref{pro1.1}, one has that $u_r$ is a maximum point of $\Phi$ restricted to $\partial\tilde{\mathcal{A}}_r$. From $(\Psi_3)$ and Lagrange multipliers theorem (see, for example, \cite[Proposition 14.3]{K}), there exists $\lambda_r\in\R$ such that
$$ \Phi'(u_r)v = \lambda_r\Psi'(u_r)v,
\hspace{0.1cm}\mbox{for all}\hspace{0.1cm}v\in E. 
$$
As $(\Psi_3)$ occurs, by choosing $v=u_r$, we obtain
\begin{equation}\label{lambdar}
\lambda_r = \frac{\Phi'(u_r)u_r}{\Psi'(u_r)u_r}.
\end{equation}
On the other hand, from \eqref{energyfunction}, the same function $u_r$ is also the minimum of the functional $J$ restricted to $\partial\tilde{\mathcal{A}}_r$. So, again, from $(\Psi_3)$ and Lagrange multipliers theorem, there exists $\gamma_r\in\R$ such that
\begin{equation}\label{fois}
J'(u_r)v = \gamma_r\Psi'(u_r)v,\,\,\mbox{for all}\,\,v\in E, 
\end{equation}
or equivalently,
$$ 
\Psi'(u_r)v - \Phi'(u_r)v = \gamma_r\Psi'(u_r)v,\,\,\mbox{for all}\,\,v\in E. 
$$
By choosing $v = u_r$ and dividing the expression by $\Psi'(u_r)u_r$, we arrive at
$$
1 - \frac{\Phi'(u_r)u_r}{\Psi'(u_r)u_r} = \gamma_r. 
$$
Finally, employing (\ref{lambdar}), we obtain the equality
\begin{equation}\label{gammar}
\gamma_r = 1 - \lambda_r.
\end{equation}
By Proposition \ref{provardif} and definition of $\zeta$, 
$$
\zeta'(r) = 1 - \frac{\Phi'(u_r)u_r}{\Psi'(u_r)u_r}. 
$$
In other words,
\begin{equation}\label{metp}
\zeta'(r) = \frac{J'(u_r)u_r}{\Psi'(u_r)u_r}. 
\end{equation}
Therefore, from equality above, if $u_r$ is a critical point of $J$, then $\zeta'(r) = 0$. Now suppose that $r$ is a critical point of $\zeta$, that is, $\zeta'(r) = 0$. Then
$$
\frac{\Phi'(u_r)u_r}{\Psi'(u_r)u_r} = 1. 
$$
Using (\ref{lambdar}), one gets
$$
\lambda_r = 1.
$$
Then, by (\ref{gammar}),
$$
\gamma_r = 0 
$$
which combines with \eqref{fois} to give 
$$
J'(u_r)v = 0,\,\mbox{for all}\, v\in E,
$$
showing that $u_r$ is a critical point for $J$. 
\end{proof}

\begin{remark}
A careful analysis of previous proof, more specifically, of equality \eqref{metp}, gives us a sufficient condition, in order a point in the Nehari set $\mathcal{N}=\{u\in E\backslash\{0\}: J'(u)u=0\}$ be a critical point of $J\in (\mathcal{J})$. In fact, if $u\in \mathcal{N}$ is a maximum energy type point in $\mathcal{G}_r$, then, by \eqref{metp}, we have $\zeta'(r)=0$. Now, from Theorem \ref{TME}, $u$ is a critical point of $J$.
\end{remark}


\subsection{Application to a semilinear problem}\label{subsection}

Let us consider the following class of semilinear problems
\begin{equation}\label{P2}
\left\{ \begin{array}{lll}
\displaystyle -\Delta u = f(x,u) &\mbox{in}&\Omega, \\
\hspace{0.7cm}u = 0 &\mbox{on}& \partial\Omega, \end{array} \right.
\end{equation}
where $\Omega\subset\R^N (N\geq 2)$ is a smooth bounded domain and $f:\Omega\times\R\rightarrow\R$ is a Carath\'eodory function.

\medskip

The classical result of Hammerstein \cite{Ham} proves that if $f$ satisfies a linear growth condition, and
$$
\limsup_{|t|\to\infty}2F(x,t)/t^2<\lambda_1, \ \mbox{uniformly in $x$},
$$
then problem \eqref{P2} has a solution, where $\lambda_1$ is the first eigenvalue of laplacian operator with Dirichlet boundary condition. Some years later, in 1986, Mawhin, Ward and Willem \cite{MWW} improve Hammerstein's result by allowing the previous upper limit to ``touch'' the first eigenvalue of laplacian, except on a set of positive measure, that is, the authors  consider that there exists $\theta\in L^{\infty}(\Omega)$ such that
\begin{equation}\label{MWW}
\limsup_{|t|\to\infty}2F(x,t)/t^2\leq \theta(x)\leq \lambda_1, \ \mbox{uniformly in $x$},
\end{equation}
with $\theta(x)<\lambda_1$ on subset of $\Omega$ of positive measure. The main goal of this section is to complement the works previously mentioned, by using the method (MEF). To be more precise, we are going to show that if $f$ is nonnegative in $(0, \infty)$ and positive for small values of $t>0$, then hypothesis \eqref{MWW} can be relaxed.

\medskip

In this section, for each $\beta\in L^{s}(\Omega)$ with $s>N/2$, $\lambda_{1}(\beta)$ stands for the first eigenvalue of 
\begin{equation}\label{Eigen}
\left\{ \begin{array}{lll}
\displaystyle -\Delta u = \lambda\beta(x)u &\mbox{in}&\Omega, \\
\hspace{0.7cm}u = 0 &\mbox{on}& \partial\Omega. \end{array} \right.
\end{equation}

\medskip
 
Function $f$ satisfies the following assumptions:

\begin{itemize}
\item[$(f_1)$] There exists $C > 0$ such that
$$ \vert f(x,t)\vert \leq C\left(1 + \vert t\vert^{p-1}\right),
\,\mbox{for all}\,t\in(0,\infty)\,\mbox{and a.e. in \,$\Omega$}, $$
with $1 < p < 2^{\ast}$ if $N\geq 3$ and $p > 1$ if $N=2$;

\medskip

\item[$(f_2)$] $f$ is nonnegative in $(0, \infty)$ and positive in $(0,\varepsilon)$, for some $\varepsilon>0$;

\medskip

\item[$(f_3)$] There exist functions $\alpha_{\ast}$, $\alpha^{\ast}$ and $\eta$ such that

\begin{eqnarray*}
\left\{ \begin{array}{lll}
\alpha^{\ast}, \eta\in L^{\infty}(\Omega) \ \mbox{and $\lambda_{1}(\alpha_\ast)<1\leq\lambda_{1}(\eta)$} &\mbox{if}& f(x, 0)=0, \\
\hspace{0.7cm} \eta\in L^{\infty}(\Omega) \mbox{ and $1\leq\lambda_{1}(\eta)$} &\mbox{if}& f(x, 0)>0, \end{array} \right.
\end{eqnarray*}
with
$$
\alpha_{\ast}(x)=\liminf_{t\rightarrow 0^+}2F(x,t)/t^{2}, \ \alpha^{\ast}(x)=\limsup_{t\rightarrow 0^+}2F(x,t)/t^{2}, \ \mbox{uniformly in $x$},
$$ 
and
$$
 \eta(x)=\limsup_{t\rightarrow\infty}2F(x,t)/t^2,  \ \mbox{uniformly in $x$}.
$$

\end{itemize}

\medskip

Some considerations are now necessary: 

\begin{enumerate}
\item Observe that, from $(f_2)$, the primitive $F^{\ast}(t)=\int_{0}^{t}f^{\ast}(x,s)\,ds$ is positive and nondecreasing for all $t>0$. Thus, $0\leq \alpha_{\ast}\leq \alpha^{\ast}$ and $\eta\geq 0$. Consequently, $(f_3)$ implies that $\alpha_{\ast}\in L^{\infty}(\Omega)$ and, by \cite{Djairo}, there exist positive first eigenvalues $\lambda_{1}(\alpha_\ast)$ and $\lambda_{1}(\eta)$. 

\medskip

\item If $f(x, 0)>0$, then the generalized L'Hospital rule implies that $\alpha_\ast=\alpha^\ast=\infty$. In fact, in this case
$$
\infty=\liminf_{t\rightarrow 0^+}f(x, t)/t\leq \liminf_{t\rightarrow 0^+}2F(x, t)/t^{2}=\alpha_{\ast}(x).
$$

\medskip

\item It is important to point out that inequality $1\leq \lambda_{1}(\eta)$ in $(f_3)$ is considerably weaker than \eqref{MWW}. In fact, it is a consequence of Proposition 1.12A in \cite{Djairo} that \eqref{MWW} implies $1<\lambda_{1}(\eta)$. On the other hand, there exist functions $\eta$ which satisfy $(f_3)$ and are greater than $\lambda_1$ in some subset $\Omega_{0}\subset\Omega$ with positive measure (and so, they are not covered by hypothesis \eqref{MWW} in \cite{MWW}). In effect, for each $k\in\N$, let us consider the set of positive measure
$$
\Omega_{k}=\{x\in\Omega: dist(x, \partial\Omega)<1/k\}
$$
and 
$$
\eta_{k}(x)=\frac{\lambda_1}{2} \chi_{\Omega\backslash\Omega_{k}(x)}+\frac{3\lambda_1}{2}\chi_{\Omega_{k}}(x),
$$
where $\chi_{A}$ denotes the characteristic function of a measurable set $A\subset\Omega$. It is clear that $\eta_k\in L^{\infty}(\Omega)$ and 
\begin{equation}\label{exem1}
\eta_{k}\to\eta_{\infty}=\frac{\lambda_1}{2} \ \mbox{in $L^{N/2}(\Omega)$}
\end{equation}
with 
\begin{equation}\label{exem2}
\lambda_{1}(\eta_{\infty})=2>1.
\end{equation} 
It follows from \eqref{exem1}, \eqref{exem2} and \cite{Djairo}, that
\begin{equation}\label{exem3}
\lambda_{1}(\eta_k)\to2.
\end{equation}
Thus, from  \eqref{exem3} there exists $k_{\ast}\in \N$ such that, defining $\eta_{\ast}:=\eta_{k_\ast}$, we have 
\begin{equation}\label{exem4}
\lambda_{1}(\eta_{\ast})>1.
\end{equation}
Now, consider for example the function
\begin{equation}
f(x, t)=\left\{ \begin{array}{lll}
-f(x, -t), &\mbox{if}& t < 0, \\
\eta_{\ast}(x)t^{2} &\mbox{if}& 0\leq t<1/4, \\
\eta_{\ast}(x)\frac{t^{3}}{(3/16)+t^{2}} &\mbox{if}& t\geq 1/4. \end{array} \right.
\end{equation}
We have that $f$ satisfies $(f_1)-(f_3)$ with $\alpha_\ast=\alpha^\ast=0$ and $\eta=\eta_{\ast}$ (see \eqref{exem4}). Moreover,
$$
\eta(x)=\frac{3\lambda_1}{2}>\lambda_1 \ \mbox{in $\Omega_{k_{\ast}}$}.
$$ 
\end{enumerate}

\medskip

In order to find a solution for \eqref{P2}, we will first prove the existence of a solution to the auxiliary problem
\begin{equation}\label{P2ast}
\left\{ \begin{array}{lll}
\displaystyle -\Delta u = f^{\ast}(x,u) &\mbox{in}&\Omega, \\
\hspace{0.7cm}u = 0 &\mbox{on}& \partial\Omega, \end{array} \right.
\end{equation}
where
\begin{equation}
f^{\ast}(x,t) = \left\{ \begin{array}{lll}
f(x, 0), &\mbox{if}& t < 0, \\
f(x,t), &\mbox{if}& t \geq 0. \end{array} \right. 
\end{equation}
For sure, $f^{\ast}$ satisfies $(f_1)-(f_3)$. 

\medskip

The functional $J$ associated with the problem (\ref{P2ast}) is given by
$$ J(u) = \Psi(u) - \Phi(u), $$
where
$$ \Psi(u) = \frac{1}{2}\Vert u\Vert^2\hspace{0.5cm}
\mbox{and}\hspace{0.5cm}\Phi(u) = \int_{\Omega}F^{\ast}(x,u)dx. $$

Now, we are going to show that $J$ belongs to class $(\mathcal{J})$ in $H_{0}^{1}(\Omega)$. In fact, by norm properties, hypothesis $(\Psi_1)$ is clearly satisfied. Moreover, by considering the change of variable $r = \frac{1}{2}\rho^2$, we get $\tilde{\mathcal{A}}_r = B_{\rho}$ and $\partial\tilde{\mathcal{A}}_r = S_{\rho}$, where $B_{\rho}$ and $S_{\rho}$ denote, respectively, the closed ball and the sphere of radius $\rho$ and centered at the origin of $H_{0}^{1}(\Omega)$. Thus, the hypotheses $(\Psi_2)-(\Psi_3)$ are certainly satisfied, $(\Psi_4)$ holds with $t_{u}(r)=r/\|u\|$, and $(\Psi_5)$ follows from the uniform convexity of $H_{0}^{1}(\Omega)$. 

\medskip

In the next lemmas, we prove that $\Phi$ verifies $(\Phi_1)-(\Phi_3)$.

\medskip

\begin{lemma}\label{lemaphi1} The functional $\Phi$ satisfies $(\Phi_1)$.
\end{lemma}

\begin{proof}
Suppose that
$$
u_n \rightharpoonup u\hspace{0.2cm}\mbox{in}\hspace{0.2cm}H_0^1(\Omega). 
$$
By compact embeddings, up to a subsequence, one has 
$$ 
u_n \rightarrow u \hspace{0.2cm}\mbox{in}\hspace{0.2cm}L^p(\Omega),
$$
$$
u_n(x) \to u(x) \quad \mbox{a.e. in} \quad \Omega,
$$
and there exists $g \in L^{p}(\Omega)$ such that
$$
|u_n| \leq g(x), \quad \mbox{a.e. in} \quad \Omega. 
$$
Therefore, 
\begin{equation}\label{2.5}
F^{\ast}(x,u_n(x)) \rightarrow F^{\ast}(x,u(x))\hspace{0.2cm}\mbox{a.e.}\hspace{0.1cm}\mbox{in}
\hspace{0.2cm}\Omega,
\end{equation}
and by $(f_1)$, 
$$ 
\vert F^{\ast}(x,u_n(x))\vert \leq C\left(\vert u_n(x)\vert + \frac{1}{p}\vert u_n(x)\vert^p\right) \leq C\left(g(x) + \frac{1}{p} g(x)^p\right)
\quad \mbox{a.e. in} \quad \Omega. 
$$
Hence, by  Lebesgue's dominated convergence theorem 
$$
\Phi(u_n) = \int_{\Omega} F^{\ast}(x,u_n)dx \rightarrow \int_{\Omega} F^{\ast}(x,u)dx = \Phi(u), 
$$
showing that $\Phi$ is weakly continuous. Since $u\mapsto \Phi'(u)u$ is also continuous and nonnegative, $\Psi'(u)u=\|u\|^{2}$ is positive and weakly lower semicontinuous, we conclude that $u\mapsto \Phi'(u)u/\Psi'(u)u$ is weakly upper semicontinuous.

\end{proof}

\begin{lemma}\label{lemaphi2} The functional $\Phi$ satisfies $(\Phi_2)$.
\end{lemma}

\begin{proof}
Suppose that $\Phi'(u) = 0$, that is,
$$
\int_{\Omega}f^{\ast}(x,u)v dx = 0,\hs\mbox{for all}\hs v\in H_0^1(\Omega). 
$$
By \cite[Corollary 4.24]{Br}, $f^{\ast}(x,u) = 0$ a.e. in $\Omega$, and so, by ($f_2$), we must have $u\leq 0$. From this, $F^{\ast}(x,u)\leq 0$ and $\Phi$ verifies ($\Phi_2$).

\end{proof}

\begin{lemma}\label{lemaphi3} The functional $\Phi$ satisfies $(\Phi_3)$.
\end{lemma}

\begin{proof}

Since $F^{\ast}(x,t)$ is positive in $\Omega\times(0,\infty)$, it is enough to choose a positive function $v\in H_0^1(\Omega)$ and define $u_n = (1/n)v$. It is clear that $u_{n}\to 0$ in $H_{0}^{1}(\Omega)$ and $\Phi(u_{n})>0$ for all $n\in\N$.

\end{proof}

The next result ensures that any solution for (\ref{P2ast}) is also a solution for \eqref{P2}.

\begin{lemma}\label{lemaposit}
Let $r = \frac{1}{2}\rho^2 $. If
$$ u_0 \in \mathcal{G}_r = \left\lbrace u\in B_{\rho};
\hspace{0.1cm}\Phi(u) = \max_{u\in S_{\rho}}\int_{\Omega}F^{\ast}(x,u)dx\right\rbrace, $$
then $u_0\geq 0$ a.e.$\hs$in $\Omega$.
\end{lemma}

\begin{proof}
Supposing by contraction that $u_0\not\equiv\vert u_0\vert$, then we must have $|[u_0<0]>0$. Since 
$$
\int_{[u_0<0]}F^{\ast}(x,u_0)dx < \int_{[u_0<0]}F^{\ast}(x,\vert u_0\vert)dx
$$
and 
$$
\int_{[u_0 \geq 0]}F^{\ast}(x,u_0)dx = \int_{[u_0 \geq 0]}F^{\ast}(x,\vert u_0\vert)dx,
$$
it follows that
$$ \Phi(u_0)= \int_{\Omega}F^{\ast}(x,u_0)dx < \int_{\Omega}F^{\ast}(x,\vert u_0\vert)dx=\Phi(|u_0|), $$
with $|u_{0}|\in S_{\rho}$. Since $u_{0}$ maximizes $\Phi$ on $S_{\rho}$, we get a contradiction. Therefore, $u_0 = \vert u_0\vert$.
\end{proof}

\medskip

\begin{theorem}
Under hypotheses $(f_1)-(f_3)$, problem \eqref{P2} has a nonnegative and nontrivial weak solution.
\end{theorem}

\begin{proof}
Since $J$ is of class $(\mathcal{J})$ in $H_{0}^{1}(\Omega)$, the energy function $\zeta:[0,\infty)\rightarrow\R$ is given by
$$
\zeta(r) = r - \max_{u\in \partial\tilde{\mathcal{A}}_{r}}\int_{\Omega}F^{\ast}(x,u)dx.
$$
Let $\sigma(\rho)=\zeta((1/2)\rho^{2})$. Then,
$$
\sigma(\rho) = \frac{1}{2}\rho^2 - \max_{u\in S_{1}}\int_{\Omega}F^{\ast}(x,\rho u)dx,
$$
and, by Proposition \ref{provardif}, it is $C^1$ in $(0,+\infty)$. Moreover, since $\sigma'(\rho)=\zeta'((1/2)\rho^{2})\rho$, it follows that if $\rho>0$ is a critical point of $\sigma$, then $r=(1/2)\rho^{2}$ is a critical point of $\zeta$.

\medskip

In order to apply Theorem \ref{TME}, we will show that $\zeta$ has a critical point. Have this in mind, let $u_\ast$ be a positive function with $\Vert u_\ast\Vert = 1$. Then
$$
\frac{\sigma(\rho)}{\rho^2}
= \frac{1}{2} - \max_{u\in S_1}\int_{\Omega}\frac{F^{\ast}(x,\rho u)}{\rho^2}dx
\leq \frac{1}{2} - \int_{\Omega}\frac{F^{\ast}(x,\rho u_\ast)}{(\rho u_\ast)^2}u_\ast^2dx. 
$$
Employing $(f_3)$ and Fatou's Lemma, one gets
\begin{equation}\label{fatou1}
\limsup_{\rho\rightarrow 0^+}\frac{\sigma(\rho)}{\rho^2} \leq \frac{1}{2} -  \liminf_{\rho\rightarrow 0^+}\int_{\Omega}\frac{F^{\ast}(x,\rho u_\ast)}{(\rho u_\ast)^2}u_\ast^2dx
\leq \frac{1}{2} - \frac{1}{2}\int_{\Omega}\alpha_{\ast}(x)u_\ast^2dx.
\end{equation}
Now, we have two cases to consider. If $f(x, 0)>0$, then $\alpha_{\ast}=\infty$ and \eqref{fatou1} implies
$$
\limsup_{\rho\rightarrow 0^+}\frac{\sigma(\rho)}{\rho^2}=-\infty.
$$
By other side, if $f(x, 0)=0$ then $\alpha_{\ast}\in L^{\infty}(\Omega)$ and, in this case, again by $(f_3)$ and replacing $u_\ast$ by the positive eigenfunction $e_1$ of $\lambda_1(\alpha_\ast)$ such that $\Vert e_1\Vert = 1$, we get from \eqref{fatou1}
$$
\limsup_{\rho\rightarrow 0^+}\frac{\sigma(\rho)}{\rho^2} \leq\frac{1}{2}\left(1-\frac{1}{\lambda_{1}(\alpha_{\ast})}\right)<0.
$$
In any case, $\sigma$ is negative for $\rho$ small enough. As $\sigma(0)=0$, our next step is to prove that for each $\varepsilon>0$ there exists $\tilde{\rho}>0$ such that $\sigma(\rho)>-\varepsilon$ if $\rho>\tilde{\rho}$, because this is enough to conclude the existence of a critical point for $\sigma$. For that, let us consider a sequence $\{\rho_n\}\subset (0,\infty)$ with  $\rho_n\rightarrow \infty$, $\{v_n\}\subset H_{0}^{1}(\Omega)$, $\|v_{n}\|=1$ and $\{\rho_{n}v_{n}\}\subset \mathcal{G}_{\rho_{n}}$. Thus 
$$ 
\frac{\sigma(\rho_n)}{\rho_n^2}
= \frac{1}{2} - \int_{[v_n>0]}\frac{F^{\ast}(x,\rho_nv_n)}{(\rho_nv_n)^2}v_n^2dx, 
$$
that is,
\begin{equation}\label{vnposz}
\frac{\sigma(\rho_n)}{\rho_n^2} =
\frac{1}{2} - \int_{\Omega}\frac{F^{\ast}(x,\rho_nv_n)}{(\rho_nv_n)^2}v_n^2\chi_{[v_n > 0]}dx.
\end{equation}
Since $\lbrace v_n\rbrace$ is a bounded sequence with $v_{n}\geq 0$ (see Lemma \ref{lemaposit}), for some subsequence, we derive that there exists $v_0 \in H_{0}^{1}(\Omega)$, with $v_{0}\geq 0$ and $\|v_0\|\leq 1$, such that
$$
v_n\rightharpoonup v_0\hs\hs\mbox{in}\hs\hs H_0^1(\Omega), 
$$
$$
v_n\rightarrow v_0\hs\hs\mbox{in}\hs\hs L^2(\Omega),
$$
$$
v_n(x)\rightarrow v_0(x)\hs\hs\mbox{a.e.}\hs\hs\mbox{in}\hs\hs\Omega
$$
and there exists $h \in L^{2}(\Omega)$ such that
$$
|v_n| \leq h(x)\hs\hs\mbox{a.e.}\hs\hs\mbox{in}\hs\hs\Omega.
$$
Since $F^{\ast}$ is nonnegative, conditions $(f_1)$ and $(f_3)$ yield that there exists $K>0$ such that 
\begin{equation}\label{v0cz}
\left\vert\frac{F^{\ast}(x,t)}{t^2}\right\vert \leq K, \quad \forall t\in\R \quad \mbox{and a.e. in } \quad \Omega. 
\end{equation}
From (\ref{v0cz}), 
\begin{equation}\label{v0chz}
\left\vert\frac{F^{\ast}(x,\rho_nv_n(x))}{(\rho_nv_n(x))^2}v_n^2(x)\chi_{[v_{n}>0]}(x)\right\vert \leq |h(x)|^2K
\hs\hs\mbox{a.e.}\hs\hs\mbox{in}\hs\hs\Omega.
\end{equation}
If $v_0=0$, the compactness Sobolev embedding together with (\ref{v0cz}) gives
$$
\liminf_{n\rightarrow \infty}\frac{\sigma(\rho_n)}{\rho_n^2} = \frac{1}{2}>0.
$$
Since $F$ is positive, for the case $v_0 \not= 0$, observe that, 
$$
 \liminf_{n\rightarrow \infty}\frac{\sigma(\rho_n)}{\rho_n^2} \geq \frac{1}{2} -
\limsup_{n\rightarrow \infty}\int_{[v_{0}> 0]}\left[\frac{F^{\ast}(x,\rho_nv_n)}{(\rho_nv_n)^2}\right]\chi_{[v_{n}> 0]}(x)v_n^2dx
$$
Since 
$$
\chi_{[v_{n}> 0]}(x)\to 1 \ \mbox{on $[v_0> 0]$},
$$ 
we employ (\ref{v0chz}) together with the version of Fatou's lemma for the upper limit to conclude that
$$
 \liminf_{n\rightarrow \infty}\frac{\sigma(\rho_n)}{\rho_n^2} \geq \frac{1}{2} - \frac{1}{2}\int_{\Omega}\eta(x)v_0^2dx.
$$
Thus, by $(f_3),$
\begin{equation}\label{fatou2*}
\liminf_{n\rightarrow \infty}\frac{\sigma(\rho_n)}{\rho_n^2}\geq \frac{1}{2}\left(1 - \frac{1}{\lambda_{1}(\eta)}\right)\|v_{0}\|^{2}\geq 0.
\end{equation}
Therefore, $\sigma$ has a critical point $\rho_{\ast} > 0$. Consequently, $r_{\ast}=(1/2)\rho_{\ast}^{2}$ is a critical point of $\zeta$ and, by Theorem \ref{TME}, there exists a maximum energy type point $u_{r_{\ast}}\in \mathcal{G}_{r_{\ast}}$ which is a nontrivial and nonnegative solution of problem (\ref{P2ast}). Since $f$ and $f^{\ast}$ coincide for nonnegative values, we conclude that $u_{r_{\ast}}$ is also a solution of problem \eqref{P2}.
\end{proof}

\section{A version of the Mountain Pass Theorem}\label{se:mountain}

In this section, we will provide a version of the mountain pass theorem by using the Method (MEF).

\begin{theorem}\label{TPM}
	Let $E$ be a reflexive Banach space and $J$ be a functional of class $(\mathcal{J})$ in $E$. Suppose that there exist $\alpha, \rho > 0$ and $w\in E$, with $w\in \mathcal{G}_{R}$, such that
	\begin{itemize}
		\item[$(H_1)$] $J(u) \geq \alpha > J(0) $ for all $u\in \partial \tilde{\mathcal{A}}_{\rho}$;
		\item[$(H_2)$] $J(w) < \alpha,$ with $R > \rho$.
	\end{itemize}
	
	Then there holds the equality below 
	\begin{equation}\label{levels} 
	c_{\ast} = \max_{r\in[0,R]}\min_{u \in \partial \tilde{\mathcal{A}}_r}J(u) \leq \inf_{\gamma\in\Gamma}\max_{t\in[0,1]}J(\gamma(t)) = c, \end{equation}
	where $\Gamma = \lbrace \gamma\in C([0,1],E); \gamma(0)=0, \gamma(1)=w\rbrace$. Moreover, $c_{\ast}>\max\{J(0),J(w)\}$ is a critical value of $J$.
\end{theorem}
\begin{proof} 
Since $J$ is a functional of class $(\mathcal{J})$, the energy function 
$$
\zeta(r) = \min_{u \in \partial \tilde{\mathcal{A}}_r} J(u) 
$$
is well-defined. By assumption $(H_1)$, 
$$
\zeta(\rho) = \min_{u \in \partial \tilde{\mathcal{A}}_\rho}J(u) > J(0) = \zeta(0). 
$$
On the other hand, from $(H_2)$,
$$
\zeta(R) = \min_{u \in \partial \tilde{\mathcal{A}}_R} J(u) \leq J(w) < \min_{u \in \partial \tilde{\mathcal{A}}_\rho}J(u) = \zeta(\rho).  
$$
So, due to the geometry of $\zeta$, there exists $r_{\ast}\in (0,R)$ such that
$$
\zeta(r_{\ast}) = \max_{r\in(0,R)}\zeta(r) = \max_{r\in[0,R]}\zeta(r)\geq \zeta(\rho)> \max\{J(0), J(w)\} \quad \mbox{and} \quad  \zeta'(r_{\ast})=0
$$
Hence, by Theorem \ref{TME}, the number $c_{\ast}$ given below 
$$ 
\max\{J(0), J(w)\} <c_{\ast} = \max_{r\in(0,R)}\zeta(r) = \max_{r\in[0,R]}\zeta(r)
= \max_{r\in[0,R]}\min_{u \in \partial \tilde{\mathcal{A}}_r}J(u). 
$$
is a critical value of $J$.

\medskip

Next, we will prove that
$$
c_{\ast} = \max_{r\in[0,R]}\min_{u \in \partial \tilde{\mathcal{A}}_r}J(u) \leq  \inf_{\gamma\in\Gamma}\max_{t\in[0,1]}J(\gamma(t)) = c. 
$$
Consider a path $\gamma\in\Gamma$ and the critical point $r_{\ast}\in(0,R)$ given above. Since $\gamma(0)=0\in \tilde{\mathcal{A}}_{r_{\ast}}$, $\gamma(1)=w\not\in \tilde{\mathcal{A}}_{r_{\ast}}$ (otherwise, we would have $J(w)\geq c_\ast$) and $\gamma([ 0,1])$ is connected, we have that $\gamma([0,1])\cap \partial \tilde{\mathcal{A}}_{r_{\ast}}$ is a nonempty set. So,
$$
 \max_{r\in[0,R]}\zeta(r) = \zeta(r_{\ast})=\min_{u \in \partial \tilde{\mathcal{A}}_{r_{\ast}}}J(u)
\leq \max_{t\in[0,1]}J(\gamma(t)),\,\mbox{for all}\,\gamma\in\Gamma, 
$$
from where it follows that 
$$
 \max_{r\in[0,R]}\zeta(r) \leq \inf_{\gamma\in\Gamma}\max_{t\in[0,1]}J(\gamma(t)) = c,
$$
or equivalently,
\begin{equation} \label{c*1}
c_{\ast} = \max_{r\in[0,R]}\min_{u \in \partial \tilde{\mathcal{A}}_{r}}J(u)
\leq \inf_{\gamma\in\Gamma}\max_{t\in[0,1]}J(\gamma(t)) = c. 
\end{equation}
\end{proof}

Next result gives a sufficient condition for the equality $c=c_\ast$.

\begin{corollary}\label{coincide}
Let $E$ be a reflexive Banach space and $J$ be a functional of class $(\mathcal{J})$ in $E$ satisfying assumptions $(H_1)-(H_2)$ given in Theorem \ref{TPM}. If the mountain pass level $c$ coincides with the infimum of $J$ on the Nehari manifold $\mathcal{N}$ associated to $J$, that is,
$$
c=\inf_{u\in \mathcal{N}}J(u), \ \mbox{where} \ \mathcal{N}=\{u\in E\backslash\{0\}: J'(u)u=0\},
$$
then $c_\ast=c=\inf_{u\in \mathcal{N}}J(u)$.

\end{corollary}

\begin{proof}
Since $c_{\ast}$ is a critical value of $J$, from Theorem \ref{TPM}, there exists $u_\ast\in E$ such that $J(u_\ast)=c_\ast> J(0)$. Therefore, $u_\ast\neq 0$, and $u_\ast\in \mathcal{N}$. Showing that $c_\ast\geq \inf_{u\in \mathcal{N}}J(u)$. The result now follows from inequality \eqref{levels}.
\end{proof}

%

Situations where $c$ coincides with $\inf_{u\in \mathcal{N}}J(u)$ are well known in the literature and cover a large class of concrete problems in applications. In these particular cases, Corollary \ref{coincide} provides a new characterization for the mountain pass level constant $c$, then it can be seen as a maxmin level. 

\medskip

Another possible advantage to have a ``mountain pass theorem'' for maxmin critical levels is in the determination of ground state critical points for $J$, that is, a critical point $w\in \mathcal{N}$ satisfying $J(w)=\inf_{u\in \mathcal{N}}J(u)$. Next result gives a step in this direction.
 
 \medskip
 
 \begin{corollary}\label{ground}
Let $E$ be a reflexive Banach space and $J$ be a functional of class $(\mathcal{J})$ in $E$ satisfying assumptions $(H_1)-(H_2)$ given in Theorem \ref{TPM}. If, 
\begin{equation}\label{condi}
J(u)\geq J(tu), \ \mbox{for all} \ t>0 \ \mbox{and} \ u\in\mathcal{N},
\end{equation}
then $J$ has a ground state critical point.

\end{corollary}

\begin{proof}
As we have seen in the proof of previous corollary, 
\begin{equation}\label{indos0}
c_\ast\geq \inf_{u\in \mathcal{N}}J(u).
\end{equation} 
To conclude the proof, it is enough to prove the opposite inequality. From Theorem \ref{TPM}, there exists a critical point $u_\ast$ for $J$ satisfying $J(u_\ast)=c_\ast=\zeta(r_\ast)$. Let $u\in \mathcal{N}$. It follows from $(\Psi_4)$ that there exists $t_{u}(r_{\ast})>0$ such that $t_{u}(r_{\ast})u\in \partial\tilde{\mathcal{A}}_{r_\ast}$. Thus, by \eqref{condi},
\begin{equation}\label{indos}
J(u)\geq J(t_{u}(r_{\ast})u)\geq J(u_{\ast})=c_\ast.
\end{equation}
where the last inequality in \eqref{indos} is a consequence of $J(u_\ast)=\min_{u\in \partial\tilde{\mathcal{A}}_{r_\ast}}J(u)$. Since $u\in \mathcal{N}$ is arbitrary in \eqref{indos}, we conclude that
\begin{equation}\label{indos1}
\inf_{u\in \mathcal{N}}J(u)\geq c_\ast.
\end{equation}
The result follows now from \eqref{indos0} and \eqref{indos1}.

\end{proof}


\section{Some applications}\label{se:app}

\subsection{A general nonlinear problem}

In this section we are going to investigate the following problem
\begin{equation}\label{PAR}
\left\{ \begin{array}{lll}
\displaystyle -\Delta u = f(x,u) &\mbox{in}&\Omega, \\
\hspace{0.7cm}u = 0 &\mbox{on}& \partial\Omega, \end{array} \right.
\end{equation}
where $\Omega\subset\R^N (N\geq 2)$ is a smooth bounded domain and $f:\Omega\times\R \to \R$ is a Carath\'eodory function satisfying the growth condition $(f_1)$, introduced in Subsection \ref{subsection}. 

\medskip 

In \cite{AR}, Ambrosetti and Rabinowitz proved existence of nontrivial solution for \eqref{PAR} by considering the following conditions on $f$: 
\begin{itemize}
\item[$(F_2)$] $\displaystyle \lim_{t\to 0}\frac{f(x, t)}{|t|}=0$;

\medskip

\item[$(AR)$] there are constants $\nu>2$ and $r\geq 0$ such that
$$
0<\nu F(x, t)\leq tf(x, t), |t|\geq r,  \ F(x,t)=\int_{0}^{t}f(x,s)\,ds.
$$
\end{itemize}

\medskip

Condition (AR) is very important to provide some compactness for Palais-Smale sequences, which is a crucial point in the Mountain Pass Theorem presented in \cite{AR}. However, despite its technical relevance, hypothesis (AR) is quite restrictive and in the last years, many authors have concentrate efforts in the search for weaker conditions that still guarantee the existence of nontrivial solution for \eqref{PAR}. Observe that (AR) implies
\begin{equation}\label{implies}
F(x, t)\geq c_{3}|t|^{\nu}-c_{4}, \ x\in\Omega, \ |t|\geq r.
\end{equation} 

In 1994, Costa and Magalh\~aes \cite{CM} studied \eqref{PAR} under the following assumptions on $f$:
\begin{itemize}

\medskip

\item[$(F_1)_q$] $\displaystyle \limsup_{|t|\rightarrow\infty}\frac{2F(x, t)}{|t|^q}\leq b<\infty$ uniformly for a.e. $x$ in $\Omega$;

\medskip

\item[$(F_1)_\mu$] $\displaystyle \liminf_{|t|\rightarrow\infty}\frac{H(x, t)}{|t|^\mu}\geq a>0$ uniformly for a.e. $x$ in $\Omega$;

\medskip

\item[$(F_3)$] $\displaystyle \limsup_{t\rightarrow 0}\frac{2F(x, t)}{t^2}\leq \alpha<\lambda_1<\beta\leq \liminf_{|t|\rightarrow \infty}\frac{2F(x,t)}{t^2}$ uniformly for a.e. $x$ in $\Omega$,
\end{itemize}

\medskip

\noindent where $H(x, t)=f(x, s)s-2F(x,t)$, $\mu>(N/2)(q-2)$ and, in practice, $q$ is smaller than $p$ in $(f_1)$. Under hypotheses $(F_1)_q$, $(F_1)_\mu$ and $(F_3)$, the authors were able to prove the existence of nontrivial solution to \eqref{PAR}. Since $\mu>2$ in (AR), it follows from \eqref{implies}, that hypothesis (AR) implies $(F_1)_\mu$ (for instance, when $q\leq 2+(4/N)$) and the last inequality in $(F_3)$.

\medskip

In 2003, Willem and Zou \cite{W} also considered a more general version of \eqref{PAR} and were able to prove existence of nontrivial solution by replacing (AR). In fact, instead (AR), the authors considered the condition

\medskip

\begin{itemize}
\item[$(WZ)$] $tf(x, t)\geq 0, \ t\in\R \ \mbox{and} \ tf(x, t)\geq c_{0}|t|^{\nu}, \ |t|\geq r$,
\end{itemize}

\medskip

\noindent where $c_{0}>$, $r\geq 0$ and $\nu>2$. Hypothesis (WZ) is certainly weaker than (AR), but it still implies \eqref{implies}.

\medskip
 
In 2004, Schechter and Zou \cite{SZ} obtained an existence result for (\ref{PAR}), under conditions $(f_1), (F_2)$ and

\medskip

\begin{itemize}
\item[$(SZ1)$] $\displaystyle \lim_{t\to \infty}\frac{F(x,t)}{t^2}=\infty$ \ uniformly in \ $x$;

\medskip

\item[$(SZ2)$]  $\nu F(x, t)- tf(x,t)\leq C(t^{2}+1)$, \ $|t|\geq r$, \ or \ $H(x, s)$ is convex in $s$,
\end{itemize} 

\medskip

\noindent for some constants $\nu> 2$ and $r\geq 0$. Hypothesis (SZ1) is much more weak than (AR) and is known in the literature as superquadraticity's condition. As observed in \cite{MS}, the first part of (SZ2) is equivalent to (AR).

\medskip

In 2008, Miyagaki and Souto \cite{MS} solved (\ref{PAR}) by assuming $f(x, 0)=0$, $(f_1)$, $(F_2)$, $(SZ1)$ and the following  monotonicity's condition: 
\noindent There exists $t_0>0$ such that

\medskip

\begin{itemize}
\item[$(MS)$] $t\mapsto f(x,t)/t$ \ \mbox{is increasing in $(t_0,\infty)$ and decreasing in $(-\infty,-t_0)$, for all $x\in\Omega$.} 
\end{itemize}

\medskip

\noindent Note that (MS) is weaker than the second part of (SZ2).

\medskip

In the present work, we are going to treat problem \eqref{PAR} by supposing that $f$ satisfies conditions $(f_{1})$ and $(f_2)$, in Subsection \ref{subsection}, and 

\medskip

\begin{enumerate}
\item[$(f_3')$] \begin{eqnarray*}
\left\{ \begin{array}{lll}
\eta\in L^{s}(\Omega) \ \mbox{with $s>N/2$, $\alpha\in L^{\infty}(\Omega)$ and $\lambda_{1}(\eta)\leq 1<\lambda_{1}(\alpha)$} &\mbox{if}&  \eta\not\equiv\infty, \\
\hspace{0.7cm} \alpha\in L^{\infty}(\Omega) \mbox{ and $1\leq\lambda_{1}(\alpha)$} &\mbox{if}&  \eta\equiv \infty, \end{array} \right.
\end{eqnarray*}
with
\begin{equation}
\alpha(x)=\limsup_{t\rightarrow 0^+}2F(x,t)/t^{2} \ \mbox{and} \ \eta(x)=\liminf_{t\rightarrow\infty}2F(x,t)/t^2, \ \mbox{uniformly in $x$}.
\end{equation}
It is a consequence of $(f_3')$ that $f(x, 0)=0$. 

\end{enumerate}

\medskip

In order to compare our assumptions with other hypotheses previously mentioned, observe that $(f_3')$ is weaker than $(F_3)$ in \cite{CM} or $(F_2)$ and $(SZ1)$ in \cite{SZ}. Moreover, $(f_2)$ (or even $(f_2')$) and $(f_3')$ together cover functions which do not satisfy (WZ) in \cite{W} and $(MS)$ in \cite{MS}. Let us consider, for instance, the Carathe\'odory function $f:\Omega\times \R\to\R$, given by
$$ 
f(x,t) = a(x)\log(t+1)(2+ \cos t)t \ \mbox{if $t>0$, and $f(x, t)=0$ if $t\leq 0$.}
$$
for some $a\in L^{\infty}(\Omega)$ with $a>0$. We have that $\alpha(x)=0$ and $\eta(x)=\infty$, therefore, $f$ satisfies $(f_1)$, $(f_2)$ and $(f_3')$, but it does not satisfy $(WZ)$ and $(MS)$. Our main theorem in this section is the following.

\begin{theorem}
Under hypotheses $(f_1)$, $(f_2)$ and $(f_3')$, problem \eqref{PAR} has a nonnegative nontrivial solution.
\end{theorem}
\begin{proof}

As in Subsection \ref{subsection}, we consider the auxiliary problem
\begin{equation}\label{P3ast}
\left\{ \begin{array}{lll}
\displaystyle -\Delta u = f^{\ast}(x,u) &\mbox{in}&\Omega, \\
\hspace{0.7cm}u = 0 &\mbox{on}& \partial\Omega, \end{array} \right.
\end{equation}
where 
$$ 
f^{\ast}(x,t) = \left\{ \begin{array}{lll}
0, &\mbox{if}& t < 0, \\
f(x,t), &\mbox{if}& t \geq 0. \end{array} \right. 
$$

\medskip

Again, due to $(f_1)-(f_2)$, as in Subsection \ref{subsection}, $F^{\ast}$ is positive and nondecreasing for all $t>0$, functional $J$ associated with the problem (\ref{P3ast}) is well defined and is of class $(\mathcal{J})$ in $H_{0}^{1}(\Omega)$, where 
$$ J(u) = \Psi(u) - \Phi(u), $$
with
$$
\Psi(u) = \frac{1}{2}\Vert u\Vert^2\hspace{0.5cm}
\mbox{and}\hspace{0.5cm}\Phi(u) = \int_{\Omega}F^{\ast}(x,u)dx. 
$$

\medskip

It remains us to prove that $J$ satisfies $(H_1)$ and $(H_2)$ given in Theorem \ref{TPM}. Indeed, let us consider a sequence $\{\rho_n\}\subset (0,\infty)$ with  $\rho_n\rightarrow 0$ and $\{v_n\}\subset H_{0}^{1}(\Omega)$ with $\|v_{n}\|=1$ and $\{\rho_{n}v_{n}\}\subset \mathcal{G}_{\rho_{n}}$. Since $(f_2)$ holds, Lemma \ref{lemaposit} is still valid here. Thus 
\begin{equation}\label{vnposz2}
\frac{\sigma(\rho_n)}{\rho_n^2} = 
\frac{1}{2} - \int_{\Omega}\frac{F^{\ast}(x,\rho_nv_n)}{(\rho_nv_n)^2}v_n^2\chi_{[v_n > 0]}dx,
\end{equation}
and there exists $v_0 \in H_{0}^{1}(\Omega)$, with $v_{0}\geq 0$ and $\|v_0\|\leq 1$, such that
\begin{equation}\label{foi}
v_n\rightharpoonup v_0\hs\hs\mbox{in}\hs\hs H_0^1(\Omega), \ v_n\rightarrow v_0\hs\hs\mbox{in}\hs\hs L^2(\Omega) \ \mbox{and} \ v_n(x)\rightarrow v_0(x)\hs\hs\mbox{a.e.}\hs\hs\mbox{in}\hs\hs\Omega,
\end{equation}
and there exists $h \in L^{2}(\Omega)$ such that
\begin{equation}\label{foisim}
|v_n| \leq h(x)\hs\hs\mbox{a.e.}\hs\hs\mbox{in}\hs\hs\Omega.
\end{equation}
Fixing $R>0$, we have
\begin{equation}\label{split}
\frac{\sigma(\rho_n)}{\rho_n^2} = 
\frac{1}{2} - \int_{[\rho_{n}v_{n}\leq R]}\frac{F^{\ast}(x,\rho_nv_n)}{(\rho_nv_n)^2}v_n^2\chi_{[v_n > 0]}dx- \frac{1}{\rho^{2}_n}\int_{[\rho_{n}v_{n}> R]}F^{\ast}(x,\rho_nv_n)dx,
\end{equation}
It follows from $(f_1)$ that there exists $q\in (2, 2^\ast)$ such that $|F^{\ast}(x, t)|\leq C|t|^{q}$ for all $|t|\geq R$ and some $C>0$. So, from  $\rho_{n}\to 0$, \eqref{foi} and \eqref{foisim}, we conclude that
\begin{equation}
\frac{1}{\rho^{2}_n}\int_{[\rho_{n}v_{n}> R]}F^{\ast}(x,\rho_nv_n)dx\to 0.
\end{equation}

On the other hand, since $F^{\ast}$ is nonnegative, conditions $(f_1)$ and $(f_3')$ yield that there exists $K>0$ such that 
\begin{equation}\label{v0cz2}
\left\vert\frac{F^{\ast}(x,t)}{t^2}\right\vert \leq K, \quad \forall \ |t|\leq R \quad \mbox{and a.e. in } \quad \Omega.
\end{equation}
From (\ref{v0cz2}) and \eqref{foisim}, 
\begin{equation}\label{v0chz2} 
\left\vert\frac{F^{\ast}(x,\rho_nv_n(x))}{(\rho_nv_n(x))^2}v_n^2(x)\chi_{[v_{n}>0]}(x)\right\vert \leq |h(x)|^2K
\hs\hs\mbox{a.e.}\hs\hs\mbox{in}\hs\hs [\rho_nv_n(x)\leq R].
\end{equation}
If $v_0=0$, the compactness Sobolev embeeding together with \eqref{split} and (\ref{v0chz2}) gives
$$
\liminf_{n\rightarrow \infty}\frac{\sigma(\rho_n)}{\rho_n^2} = \frac{1}{2}>0.
$$
For the case $v_0 \not= 0$, we employ (\ref{v0chz2}) together with versions for the upper limit of Fatou's lemma and the convergence
$$
\chi_{[v_n>0]\cap[\rho_{n}v_{n}\leq R]}(x)\to 1, \ \mbox{a.e. in $[v_0>0]$}
$$
to get 
\begin{equation}\label{fatou2}
 \liminf_{n\rightarrow \infty}\frac{\sigma(\rho_n)}{\rho_n^2} \geq  \frac{1}{2} - \frac{1}{2}\int_{\Omega}\alpha(x)v_0^2dx.
\end{equation}
Thus, by $(f_3')$
$$
 \liminf_{n\rightarrow \infty}\frac{\sigma(\rho_n)}{\rho_n^2}\geq \frac{1}{2}\left(1 - \frac{1}{\lambda_{1}(\alpha)}\right)\|v_{0}\|^{2}> 0.
 $$
 Therefore, there exists $\rho$ small enough such that $\displaystyle \min_{u\in\mathcal{S}_{\rho}}J(u)=\sigma(\rho)>0$. Showing that $(H_1)$ holds.

\medskip

Now let $e_1\in S_1$ be the positive eigenfunction associated to $\lambda_1(\eta)$. Then, 
$$
\frac{J(r e_1)}{r^2} = \frac{1}{2} - \int_{\Omega}\frac{F^{\ast}(x,re_1)}{(re_1)^2}e_1^2dx. 
$$
Since $F^{\ast}$ is a positive function when $t\in(0,\infty)$, the Fatou's Lemma gives 
$$
\limsup_{r\rightarrow\infty}\frac{J(r e_1)}{r^2}
= \frac{1}{2} - \liminf_{r\rightarrow\infty}\int_{\Omega}\frac{F^{\ast}(x,re_1)}{(re_1)^2}e_1^2dx
\leq \frac{1}{2} - \frac{1}{2}\int_{\Omega}\eta(x)e_1^2dx.
$$
If $\eta\equiv\infty$, it follows from previous inequality that
$$
\limsup_{r\rightarrow\infty}\frac{J(r e_1)}{r^2}=-\infty.
$$ 
By other side, if $\eta\not\equiv\infty$, by $(f_3')$, we conclude that
$$
\limsup_{r\rightarrow\infty}\frac{J(r e_1)}{r^2} \leq 
\frac{1}{2}\left(1 - \frac{1}{\lambda_{1}(\eta)}\right)\leq 0. 
$$
In any case, given $\underline{\alpha} > 0$ with $\underline{\alpha} < \alpha$, there exists $R > 0$ sufficiently large such that
$$ J(R e_1) \leq \underline{\alpha} < \alpha, $$
proving $(H_2)$. Thereby, by Theorem \ref{TPM}, $J$ attains a critical value in
$$ c_{\ast} = \max_{r\in[0,R]}\min_{\Vert u\Vert = r}J(u). $$

In other words, there exists $r_{\ast}\in (0,R)$ such that $u_{r_{\ast}}$ is a nontrivial solution to the problem (\ref{P3ast}). Since $f$ and $f^{\ast}$ coincide for nonnegative values, we conclude that $u_{r_{\ast}}$ is also a solution to the problem (\ref{PAP}).

\end{proof}

%
%



\subsection{A Berestycki-Lions type problem}\label{se:BL}

In 1983, Berestycki and Lions \cite{BL} were the first to solve the problem
\begin{equation}\label{PBL}
\left\{\begin{array}{lll}
\displaystyle -\Delta u = g(u)\hs\hs\mbox{in}\hs\hs\R^N, \\
\hspace{0.4cm}u \in H^1(\R^N), \end{array} \right.
\end{equation}
without using the Ambrosseti and Rabinowitz condition. However, the method used in \cite{BL} is not valid for non-autonomous problems, even though the function $g(x,u)$ is radially symmetric with respect to $x$.

\medskip

In 2009, Azzollini and Pomponio \cite{AP}, motivated by approach explored in Jeanjean \cite{Jj}, considered the problem above with $g(x,t)=f(t)-V(x)t$, which can be rewritten of the form
\begin{equation}\label{PAP}
\left\{\begin{array}{lll}
\displaystyle -\Delta u + V(x)u = f(u)\hs\hs\mbox{in}\hs\hs\R^N, \\
\hspace{0.4cm}u \in H^1(\R^N), \end{array} \right.
\end{equation}
assuming that $N\geq 1$ and $V\in C^1(\R^N,\R)$ satisfying the following conditions:
\begin{itemize}
\item[$(V_1')$] $V(x) = V(\vert x\vert)$;
\medskip
\item[$(V_2')$] $V(x)\geq 0$ for all $x\in\R^N$ and the inequality is strict somewhere;
\item[$(V_3')$] $\Vert (\nabla V(\cdot)|\cdot)^+\Vert_{N/2} < 2S$, where $\disp S = \inf_{u\in {D}^{1,2}(\R^N)\setminus\lbrace 0\rbrace}\frac{\vert\nabla u\vert_2^2}{\vert u\vert_{2 ^{\ast}}^2}$;
\item[$(V_4')$] $\disp\lim_{\vert x\vert\rightarrow\infty}V(x) = 0$.
\end{itemize}

In the present section, we intend to use Theorem \ref{TPM} to complement the main result found in \cite{AP}, in the sense we will only consider the following conditions on potential $V$. 
\begin{itemize}
\item[$(V_1)$] $V(x) = V(\vert x\vert)$;
\medskip
\item[$(V_2)$] There exist $V_0,V_{\infty} > 0$ such that $V_0 \leq V(x) \leq V_{\infty},\,\forall\, x\in\R^N$.
\end{itemize}
\medskip

By changing $(V_2')$--$(V_4')$ for $(V_2)$, it is possible to obtain functions that do not satisfy the conditions of \cite{AP}, but that satisfies our hypotheses, see for example the function
$$
V(x)=\frac{\vert x\vert + k_1}{\vert x\vert +k_2},\hs\hs\mbox{with}\hs\hs 0 < k_1 < k_2. 
$$

With respect to $f\in C(\R,\R)$, let us consider the following assumptions:
\begin{itemize}
\item[$(f_1)$]  
$$ 
\lim_{t\rightarrow 0}\frac{f(t)}{t} = 0\hspace{0.5cm}\mbox{and}\hspace{0.5cm}
\limsup_{\vert t\vert\rightarrow\infty}\frac{\vert f(t)\vert}{\vert t\vert^{p-1}}<\infty,
\,\,\mbox{where}\,\,2 < p < 2^{\ast}; 
$$
\item[$(f_2)$] There exists $t_0 > 0$ such that $F(t_0) - \frac{V_{\infty}}{2}t_0^2 > 0$;
\item[$(f_3)$] $f(t)> 0$ for all $t>0$.
\end{itemize}

In the sequel,  we will restrict the functional $J$ to the subspace $H^1_{rad}(\R^N)\subset H^1(\R^N)$, because by Strauss' Lemma \cite [Theorem 1.2, Chapter 6]{K} we have the compact embedding
\begin{equation}\label{imer2}
H^1_{rad}(\R^N)\hookrightarrow L^p(\R^N),\,\,\mbox{with}\,\, 2 < p < 2^{\ast}.
\end{equation}

As in the previous section, let us consider the auxiliary problem
\begin{equation}\label{P4ast}
\left\{ \begin{array}{lll}
\displaystyle -\Delta u + V(x)u = f^{\ast}(u) &\mbox{in}&\R^N, \\
\hspace{0.7cm}u \in H^1_{rad}(\R^N), \end{array} \right.
\end{equation}
where 
$$ 
f^{\ast}(t) = \left\{ \begin{array}{lll}
0, &\mbox{if}& t < 0, \\
f(t), &\mbox{if}& t \geq 0. \end{array} \right. 
$$
A simple computation gives that $f^{\ast}$ satisfies the same conditions as $f$.

Due to $(f_3)$, $F^{\ast}$ does not change sign in $(0,\infty)$. Then, from $(f_2)$, $F^{\ast}$ is positive and increasing for all $t>0$. The functional $J:H^1_{rad}(\R^N)\rightarrow\R$ is given by
$$ J(u) = \Psi(u) - \Phi(u), $$
where
$$
\Psi(u) = \frac{1}{2}\int_{\R^N}\left(\vert \nabla u\vert^2+V(x)\vert u\vert^2\right)dx
\hspace{0.5cm}\mbox{and}\hspace{0.5cm}\Phi(u) = \int_{\R^N}F^{\ast}(u)dx.
$$
Hereafter, we will denote by $\|\,\,\,\,\|_*$ the following norm in $H^{1}(\R^N)$
$$
\|u\|_*=\left(\int_{\R^N}(|\nabla u|^{2}+V(x)|u|^2)\,dx\right)^{\frac{1}{2}}.
$$
By condition $(V_2)$ it is easy to check that the norm above is equivalent the usual norm in $H^{1}(\R^N)$.

Next, we will show that $J$ belongs to class $(\mathcal{J})$.

\begin{lemma}
The functional $\Phi$ satisfies $(\Phi_1)$.
\end{lemma}

\begin{proof}

By $(f_1)$, given $\varepsilon > 0$, there exists a constant $C=C(\varepsilon) > 0$ such that
$$
\vert f^{\ast}(t)\vert \leq \varepsilon\vert t\vert + C\vert t\vert^{p-1}, \quad \forall t \in \mathbb{R} 
$$
From this inequality, 
\begin{equation}\label{desF}
\vert F^{\ast}(t)\vert \leq \frac{\varepsilon}{2}\vert t\vert^2 + \frac{C}{p}\vert t\vert^{p}, \quad \forall t \in \mathbb{R},
\end{equation}
where $2 < p < 2^{\ast}$. Setting the functions $P,Q:\R\rightarrow\R$ by
$$
P(t)=F^{\ast}(t)\hs\hs\hs\mbox{and}\hs\hs\hs Q(t)=t^2+ \vert t\vert^{2 ^{\ast}}, 
$$
it follows that 
\begin{equation}\label{BL1}
\lim_{\vert t\vert\rightarrow\infty}\frac{P(t)}{Q(t)} = 0
\end{equation}
and
\begin{equation}\label{BL2}
\lim_{t\rightarrow 0}\frac{P(t)}{Q(t)} = 0.
\end{equation}
Finally, assuming that $u_n \rightharpoonup u_0$ in $H^1_{rad}(\R^N)$, we derive that
\begin{equation}\label{BL3}
\sup_{n\in\N} \int_{\R^N}\vert Q(u_n(x))\vert dx < \infty .
\end{equation}
Therefore,  (\ref{BL1}), (\ref{BL2}) and (\ref{BL3}) permit to apply \cite[Theorem A.I]{BL} to get
$$
\Phi(u_n) = \int_{\R^N}F^{\ast}(u_n)dx \rightarrow \int_{\R^N}F^{\ast}(u_0)dx = \Phi( u_0), 
$$
proving that $\Phi$ is weakly continuous. It is possible to obtain the same result for $u\mapsto\Phi'(u)u$.

\end{proof}

For the other hypotheses of $\Phi$, the procedure is similar to the previous application. Moreover, it is also true that, fixed $r>0$, if
$$ u_0 \in \mathcal{G}_r = \left\lbrace u\in \tilde{\mathcal{A}}_r;
\hspace{0.1cm}\Phi(u) = \max_{u\in\partial\tilde{\mathcal{A}}_r}\int_{\R^N}F^{\ast}(u)dx\right\rbrace, $$
then $u_0\geq 0$ a.e.$\hs$in $\R^N$.

\begin{theorem}
Under hypotheses $(V_1)$--$(V_2)$ and $(f_1)$--$(f_3)$, problem (\ref{PAP}) has a nontrivial weak solution.
\end{theorem}

\begin{proof}

In what follows, we will prove that $J$ satisfies $(H_1)$ and $(H_2)$ in Theorem \ref{TPM}. Indeed, by (\ref{desF}), 
$$
J(u) \geq \Vert u\Vert_{\ast}^2 - \frac{\varepsilon}{2}\vert u\vert_2^2 - \frac{C}{p}\vert u\vert_p^p.
$$
Fixing $\varepsilon > 0$ small and using continuous Sobolev embeddings, we arrive at 
$$
J(u) \geq C_1\Vert u\Vert_{\ast}^2 - C_2\Vert u\Vert_{\ast}^p. 
$$
As $p>2$, for $\rho$ small enough, there exists $\alpha > 0$ such that
$$
J(u)\geq\alpha,\hs\mbox{for all}\hs u\in H^1_{rad}(\R^N)\hs\hs\mbox{with}\hs\hs \Vert u\Vert_{\ast}=\rho,  
$$
proving $(H_1)$. 

On the other hand, using $(f_2)$, note that there exists $\varphi\in C_0^{\infty}(\R^N)$ such that
\begin{equation}\label{F2BL}
 \int_{\R^N}F(\varphi)dx - \frac{V_{\infty}}{2}\vert \varphi\vert_2^2 > 0.
\end{equation}
By taking $\disp w(t, x)=\varphi\left(\frac{x}{t}\right)$ with $t>0$, we obtain
$$
\int_{\R^N}\vert\nabla w\vert^2dx =
\frac{1}{t^2}\int_{\R^N}\left\vert\nabla\varphi\left(\frac{x}{t}\right)\right\vert^2 dx
= t^{N-2}\int_{\R^N}\vert\nabla\varphi(x)\vert^2dx 
$$
and
$$ 
\int_{\R^N}V(x)\left\vert\varphi\left(\frac{x}{t}\right)\right\vert^2dx
= t^N\int_{\R^N}V(xt)\vert\varphi(x)\vert^2dx. 
$$
Fixing $R_t = \Vert w\Vert_{\ast}$, one gets
$$
R_t^2 = t^{N-2}\int_{\R^N}\vert\nabla\varphi(x)\vert^2dx
+ t^N\int_{\R^N}V(xt)\vert\varphi(x)\vert^2dx.  
$$
Consequently, 
$$
J(w) \leq \frac{t}{2}^{N-2}\int_{\R^N}\vert\nabla\varphi\vert^2dx
+ \frac{t}{2}^N\int_{\R^N}V_{\infty}\vert\varphi\vert^2dx - t^N\int_{\R^N}F(\varphi)dx, 
$$
that is, 
$$
J(w) \leq \frac{t}{2}^{N-2}\int_{\R^N}\vert\nabla\varphi\vert^2dx
- t^N\left(\int_{\R^N}F(\varphi)dx - \frac{V_{\infty}}{2}\int_{\R^N}\vert\varphi\vert^2dx\right). 
$$
Thus, by (\ref{F2BL}), $J(w) < 0$ for $t$ large enough. Note that $R_t\rightarrow\infty$ when $t\rightarrow\infty$, because we have the inequality below 
$$
t^{N-2}\vert\nabla\varphi\vert_2^2 + V_0t^N\vert\varphi\vert_2^2 \leq R_t^2 \leq t^{N-2}\vert\nabla\varphi\vert_2^2 + V_{\infty}t^N\vert\varphi\vert_2^2. 
$$
From this, $(H_2)$ also occurs and we can use the Theorem \ref{TPM} to conclude that $J$ attains a critical value in
$$
c_{\ast} = \max_{r\in[0,R]}\min_{\Vert u\Vert_{\ast} = r}J(u). 
$$

In other words, there exists $r_{\ast}\in (0,R)$ such that $u_{r_{\ast}}$ is a nontrivial critical point of $J$ in $H_{rad}^{1}(\R^N)$. However, we can apply the Palais Criticality Principle \cite{W} to guarantee that $u_{r_{\ast}}$ is also a critical point of $J$ in $H^1(\R^N)$ and, since that $f$ and $f^{\ast}$ coincide for nonnegative values, we deduce that $u_{r_{\ast}}$ is a solution to the problem (\ref{PAP}).

\end{proof}



\subsection{An anisotropic equation}\label{se:an}

Consider the problem
\begin{equation}\label{PAn}
\left\{ \begin{array}{lll}
\displaystyle -\sum_{i=1}^N\partial_i(\vert\partial_iu\vert^{p_i-2}\partial_iu)
= f(u) &\mbox{in}&\Omega,\\
\hspace{3.4cm}u = 0 &\mbox{on}& \partial\Omega, \end{array} \right.
\end{equation}
where $N \geq 2$, $p_i$'s are ordered , that is, $1<p_1\leq p_2\leq ...\leq p_N$ and $f:\R\rightarrow\R$ is a continuous function with $f(0)=0$. In 2009, Di Castro and Montefusco \cite{DM} solved the problem (\ref{PAn}) considering $f(u) = \lambda \vert u\vert^{q-2}u$, with $\lambda > 0$ and $p_1 < q < p_N$. Note that the function chosen by the authors already satisfies the Ambrosseti and Rabinowitz condition. So, our goal is to solve (\ref{PAn}) using a more general function.

\medskip

In our case, $f$ satisfies the following conditions:
\begin{itemize}
\item[$(f_1)$] There exists $C > 0$ such that
$$ \vert f(t)\vert \leq C\left(1 + \vert t\vert^{q-1}\right)\,\mbox{for all}\,t\in(0,\infty),\,\mbox{with}\,\,p_1 < q < p_N; $$
\item[$(f_2)$]
$$ 0 \leq \limsup_{t\rightarrow 0^+}\frac{F(t)}{\vert t\vert^{p_N}} <
\frac{\xi}{\Theta^{p_N}} \,\,\mbox{and}\,\,\,\,
\mathcal{S} \leq \liminf_{t\rightarrow\infty}\frac{F(t)}{\vert t\vert^{p_N}}, $$
where $F(x,t) = \int_0^t f(x,s)ds$ and  $\mathcal{S}, \Theta$ and $\xi$ are given in (\ref{imeran}), (\ref{S}) and (\ref{xi}) respectively.
\medskip
\item[$(f_3)$] $f(t)> 0$ for all $t>0$.
\end{itemize}

As function space we will consider the space $W_0^{1,\vec{p}}(\Omega)$ endowed with the norm
$$ 
\Vert u\Vert_{\vec{p}} = \sum_{i=1}^N\Vert\partial_iu\Vert_{p_i}, 
$$
where $\vec{p}=(p_{1}, p_{2}, \ldots, p_{N})$.

\medskip

Recall that $W_0^{1,\vec{p}}(\Omega)$ is a reflexive Banach space which is compactly embedded in $L^s(\Omega)$ for all $s\in[1,p_{\infty})$ (see \cite{Fra}), where 
$p_{\infty}=\max\lbrace p_N,\bar{p}^{\ast}\rbrace$, $\bar{p}=\displaystyle N/\left(\displaystyle\sum^{N}_{i=1}\frac{1}{p_{i}}\right)$ is the harmonic mean of the $p_i$'s and
$$
p^{\ast}:=
\displaystyle\frac{N}{\left(\displaystyle\sum^{N}_{i=1}\frac{1}{p_{i}}\right)-1}=\displaystyle\frac{N\overline{p}}{N-\overline{p}}.
$$
Moreover, we fix $\Theta, \xi>0$ satisfying  
\begin{equation}\label{imeran}
\vert u\vert_{p_N}\leq\Theta \Vert u\Vert_{\vec{p}}, \quad \forall u \in W_0^{1,\vec{p}}(\Omega)
\end{equation}
and
\begin{equation} \label{xi}
\xi \Vert u\Vert^{p_N}_{\vec{p}} \leq  \frac{1}{p_N}\sum_{i=1}^N\|\partial_i u\|^{p_N}_{p_i}, \quad \forall u \in W_0^{1,\vec{p}}(\Omega).
\end{equation}
The existence of such a constant $\Theta$ is guaranteed in \cite[Theorem 1]{Fra}.

\medskip

Hereafter, we assume that $p_N < \bar{p}^{\ast}$, and so, we have the compact embedding $W_0^{1,\vec{p}}(\Omega) \hookrightarrow L^{p_N}(\Omega)$. Moreover, we will denote by $\mathcal{S}>0$ the following constant
\begin{equation} \label{S}
\mathcal{S}=\min_{u \in W_0^{1,\vec{p}}(\Omega)}\left\{\sum_{i=1}^N\Vert\partial_iu\Vert^{p_i}_{p_i}\;:\;\|u^{+}\|_{p_N}=1\right\}.
\end{equation}

\begin{theorem}
Under hypotheses $(f_1)-(f_3)$, problem \eqref{PAn} has a nontrivial weak solution.
\end{theorem}
\begin{proof}

We will first prove the existence of a solution to the auxiliary problem
\begin{equation}\label{PAnast}
\left\{ \begin{array}{lll}
\displaystyle -\sum_{i=1}^N\partial_i(\vert\partial_iu\vert^{p_i-2}\partial_iu)
= f^{\ast}(u) &\mbox{in}&\Omega,\\
\hspace{3.4cm}u = 0 &\mbox{on}& \partial\Omega, \end{array} \right.
\end{equation}
where $f^{\ast}$ is defined by
$$ f^{\ast}(t) = \left\{ \begin{array}{lll}
0, &\mbox{if}& t < 0, \\
f(t), &\mbox{if}& t \geq 0. \end{array} \right. $$

Due to $(f_3)$, $F^{\ast}$ does not change sign in $(0,\infty)$. Then, from $(f_2)$, $F^{\ast}$ is positive and increasing for all $t>0$. The functional $J:W_0^{1,\vec{p}}(\Omega)\rightarrow\R$ associated with the problem (\ref{PAnast}) is given by
$$ 
J(u) = \sum_{i=1}^N\frac{1}{p_i}\int_{\Omega}\vert\partial_iu\vert^{p_i}dx
- \int_{\Omega}F^{\ast}(u) dx. 
$$
In the sequel
$$
\Psi(u)= \sum_{i=1}^N\frac{1}{p_i}\int_{\Omega}\vert\partial_iu\vert^{p_i}dx \quad \mbox{and} \quad \Phi(u)=\int_{\Omega}F^{\ast}(u) dx.
$$

\begin{lemma} Hypothesis $(\Psi_5)$ is satisfied in $W_0^{1,\vec{p}}(\Omega)$.
	
\end{lemma}
\begin{proof}  Let $(u_n) \in W_0^{1,\vec{p}}(\Omega)$ with $u_n \rightharpoonup u$ and $\displaystyle \lim_{n \to +\infty}\Psi(u_n)=\Psi(u)$. Then, by the weak convergence,
$$	
\sum_{i=1}^N\frac{1}{p_i}\int_{\Omega}\vert\partial_iu\vert^{p_i}dx = \limsup_{n \to +\infty}\sum_{i=1}^N\frac{1}{p_i}\int_{\Omega}\vert\partial_iu_n\vert^{p_i}dx\geq \liminf_{n \to +\infty}\sum_{i=1}^N\frac{1}{p_i}\int_{\Omega}\vert\partial_iu_n\vert^{p_i}dx\geq\sum_{i=1}^N\frac{1}{p_i}\int_{\Omega}\vert\partial_iu\vert^{p_i}dx,
$$
hence, for a some subsequence, one has
$$	
\sum_{i=1}^N\frac{1}{p_i}\int_{\Omega}\vert\partial_iu\vert^{p_i}dx= \sum_{i=1}^N\frac{1}{p_i} \limsup_{n \to +\infty}\int_{\Omega}\vert\partial_iu_n\vert^{p_i}dx\geq \sum_{i=1}^N\frac{1}{p_i}\liminf_{n \to +\infty}\int_{\Omega}\vert\partial_iu_n\vert^{p_i}dx\geq\sum_{i=1}^N\frac{1}{p_i}\int_{\Omega}\vert\partial_iu\vert^{p_i}dx,
$$
If there is $i \in \{1,...,N\}$ such that
$$
\liminf_{n \to +\infty}\int_{\Omega}\vert\partial_iu_n\vert^{p_i}dx>\int_{\Omega}\vert\partial_iu\vert^{p_i}dx,
$$
or 
$$
\limsup_{n \to +\infty}\int_{\Omega}\vert\partial_iu_n\vert^{p_i}dx>\liminf_{n \to +\infty}\int_{\Omega}\vert\partial_iu\vert^{p_i}dx,
$$
we get a contradiction. From this, 
$$
\limsup_{n \to +\infty}\int_{\Omega}\vert\partial_iu_n\vert^{p_i}dx=\liminf_{n \to +\infty}\int_{\Omega}\vert\partial_iu_n\vert^{p_i}dx=\sum_{i=1}^N\frac{1}{p_i}\int_{\Omega}\vert\partial_iu\vert^{p_i}dx, \quad \forall i \in \{1,...N\}.
$$
Using the fact that $L^{p_i}(\Omega)$ is uniformly convex and that $\partial_iu_n \rightharpoonup \partial_iu$ in $L^{p_i}(\Omega)$, we derive that
$$
\partial_iu_n \to \partial_iu \quad \mbox{in} \quad L^{p_i}(\Omega), \quad \forall i \in \{1,...N\},
$$
showing that
$$
u_n \to u \quad \mbox{in} \quad  W_0^{1,\vec{p}}(\Omega),
$$
finishing the proof. 
	
\end{proof}

Similar to what was done in the previous sections, it is possible to prove that $J$ is of class $(\mathcal{J})$ and that the functions $u_r$ are nonnegative. It remains us to prove that $J$ satisfies $(H_1)$ and $(H_2)$ in Theorem \ref{TPM}. Indeed, by $(f_1)-(f_2)$, given $\epsilon \in (0, \frac{\xi}{\Theta^{p_N}})$ there exists $C>0$ such that
$$
F^{\ast}(t) \leq \epsilon |t|^{p_N}+C|t|^{\overline{p}^*}, \quad \forall t \in \R. 
$$
$$
J(u)=\sum_{i=1}^N\frac{1}{p_i}\int_{\Omega}\vert\partial_iu\vert^{p_i}dx - \int_{\Omega}F^{\ast}(u)dx
\geq \frac{1}{p_N}\sum_{i=1}^N\int_{\Omega}\vert\partial_iu\vert^{p_i}dx-\epsilon \Theta^{p_N}\Vert u\Vert_{\vec{p}}^{p_N}\,dx-C \Vert u\Vert_{\vec{p}} ^{\overline{p}^*}, 
$$
For $\Vert u\Vert_{\vec{p}} \leq 1$, we get
$$
J(u) \geq \frac{1}{p_N}\sum_{i=1}^N\|\partial_i u\|^{p_N}_{p_i}-\epsilon \Theta^{p_N}\Vert u\Vert_{\vec{p}}^{p_N}\,dx-C \Vert u\Vert_{\vec{p}} ^{\overline{p}^*} \geq \xi\Vert u\Vert_{\vec{p}}^{p_N}-\epsilon \Theta^{p_N}\Vert u\Vert_{\vec{p}}^{p_N}\,dx-C \Vert u\Vert_{\vec{p}} ^{\overline{p}^*} , 
$$
where $\xi$ was given in (\ref{xi}). From $(f_2)$,
$$
J(u) \geq C_1\Vert u\Vert_{\vec{p}}^{p_N}-C \Vert u\Vert_{\vec{p}} ^{\overline{p}^*}.
$$
As $\overline{p}^*>p_N$, for $\rho$ small enough, there exists $\alpha > 0$ such that
$$
J(u)\geq\alpha,\hs\mbox{for all}\hs u\in W_0^{1,\vec{p}}(\Omega)\hs\hs\mbox{with}\hs\hs \Vert u\Vert_{\ast}=\rho,  
$$
proving $(H_1)$.

\medskip

Now let $v\in W_{0}^{1,\vec{p}}(\Omega)$ such that 
$$
\mathcal{S}=\sum_{i=1}^N\Vert\partial_iu\Vert^{p_i}_{p_i} \quad \mbox{and} \quad \|v^+\|_{p_N}=1
$$
and $r>0$ such that $rv \in \partial \tilde{\mathcal{A}}_{R}$ . Then, $r \to +\infty$ if, and only if $R \to +\infty$. Thus, 
$$
\frac{J(rv)}{r^{p_N}} =
\sum_{i=1}^N\frac{r^{p_i}}{r^{p_N}p_i}\int_{\Omega}\vert\partial_i v\vert^{p_i}dx
-\int_{[v>0]}\frac{F^{\ast}(rv)}{\vert rv\vert^{p_N}}\vert v\vert^{p_N}dx. 
$$
Thus, for $r>1$, 
$$
\limsup_{r \to +\infty}\frac{J(r v)}{r^{p_N}} \leq \sum_{i=1}^N\int_{\Omega}\vert\partial_i v\vert^{p_i}dx
-\mathcal{S}\int_{[v>0]}\vert v\vert^{p_N}dx=\mathcal{S}-\mathcal{S}\leq 0. 
$$
So, given $\underline{\alpha} > 0$ with $\underline{\alpha} < \alpha$, there exists $r > 0$ sufficiently large such that
$$ J(r v) \leq \underline{\alpha} < \alpha, $$
proving $(H_2)$. 

Therefore, by Theorem \ref{TPM}, there exists $r_{\ast}\in (0,R)$ such that $u_{r_{\ast}}$ is a nontrivial solution to the problem (\ref{PAnast}). Since $f$ and $f^{\ast}$ coincide for nonnegative values, we conclude that $u_{r_{\ast}}$ is also a solution to the problem (\ref{PAn}).

\end{proof}


\end{document}